\documentclass[14pt]{extarticle}

\usepackage[displaymath, mathlines,running]{lineno}

\usepackage{xcolor}
\definecolor{mycolor1}{rgb}{0.00000,0.44700,0.74100}
\definecolor{mycolor2}{rgb}{0.85000,0.32500,0.09800}%
\usepackage{mathtools}
\usepackage{algorithm}
\usepackage{mathtools}
\usepackage[noend]{algpseudocode}
\usepackage{amsthm}
\usepackage{stmaryrd}
\usepackage[margin=1.0in]{geometry}
\usepackage{titling}
\usepackage{xspace}
\usepackage{blkarray}
\usepackage{multicol}
\usepackage{bm}
\usepackage{url}
\usepackage{tikz}
\usepackage{pgfplots}
\usepackage{tikz-3dplot}
\usepackage{mathrsfs}
\usepackage{tikz-cd}
\usepackage{multirow}
\usepackage{comment}
 \allowdisplaybreaks
\usepackage{multicol}

\usepackage[font=footnotesize,labelfont=bf]{caption}
\usepackage{subfig} 

\usepackage{array}
\usepackage{enumitem}
\usepackage{authblk}
\usepackage{blindtext}
\usepackage{hyperref}
\hypersetup{
    colorlinks=true,
    linkcolor=orange!80!black,
    filecolor=magenta,      
    urlcolor=cyan,
    citecolor=blue!50!white,
  }
\usepackage{adjustbox}
\newcolumntype{R}[2]{%
    >{\adjustbox{angle=#1,lap=\width-(#2)}\bgroup}%
    l%
    <{\egroup}%
}
\makeatother
\usepackage{makecell}
\usepackage{amssymb}
\usepackage{subfig}
\usepackage{todonotes}
\makeatletter
\newcommand{\xdashrightarrow}[2][]{\ext@arrow 0359\rightarrowfill@@{#1}{#2}}
\makeatother
\let\emptyset\varnothing

\usepackage{cleveref}

\DeclareMathOperator{\BWE}{BWE}

\theoremstyle{plain}
\newtheorem{theorem}{Theorem}[section]
\newtheorem{lemma}{Lemma}[section]

\newtheorem{remark}{Remark}[section]
\newtheorem{proposition}{Proposition}[section]
\newtheorem{corollary}{Corollary}[section]

\theoremstyle{definition}
\numberwithin{equation}{section}
\newtheorem{definition}{Definition}[section]

\newtheorem{examplex}{Example}
\newenvironment{example}
  {\pushQED{\qed}\examplex}
  {\popQED\endexamplex}
\newtheorem{assumption}{Assumption}

\newcommand{\N}{\mathbb{N}}
\newcommand{\conv}{\textup{Conv}}

\newcommand{\HF}{\textup{HF}}

\newcommand{\id}{\textup{id}}

\newcommand{\Groebner}{Gr\"{o}bner\xspace}

\newcommand{\x}{x}
\newcommand{\f}{\mathcal{F}}

\newcommand{\diag}{\textup{diag}}

\newcommand{\crt}{\texttt{crt}}
\newcommand{\F}{\mathcal{F}}

\newcommand{\I}{\mathcal{I}}
\renewcommand{\Im}{\mathrm{Im}}

\newcommand{\Sylv}{\textup{Sylv}}
\newcommand{\Mat}{\textrm{M}}
\newcommand{\B}{{B}}

\newcommand{\Cok}{\textup{Coker}}

\newcommand{\rank}{\textup{rank}}

\newcommand{\M}{\mathcal{M}}

\newcommand{\Z}{\mathbb{Z}}

\newcommand{\C}{\mathbb{C}}

\newcommand{\z}{\zeta}

\newcommand{\D}{\delta}

\renewcommand{\z}{\zeta}

\newcommand{\R}{\mathbb{R}}

\providecommand{\keywords}[1]{\small \textbf{Key words ---} #1}
\providecommand{\classification}[1]{\small \textbf{AMS subject classifications ---} #1}
\newcommand{\pair}[1]{\langle{#1}\rangle}
\newcommand{\ideal}[1]{\left \langle {#1}  \right \rangle}

\newcommand{\link}{{\url{https://github.com/simontelen/JuliaEigenvalueSolver}}}

\newlength{\bibitemsep}
\newlength{\bibparskip}
\let\oldthebibliography\thebibliography
\renewcommand\thebibliography[1]{%
  \oldthebibliography{#1}%
  \setlength{\parskip}{\bibitemsep}%
  \setlength{\itemsep}{\bibparskip}%
}

\usepackage{mathptmx}
\DeclareMathAlphabet{\mathcal}{OMS}{cmsy}{m}{n}

\begin{document}
\title{Yet another eigenvalue algorithm for solving polynomial systems}
\author[ ]{Mat{\'\i}as R. Bender\thanks{Department of Mathematics, Technische Universit\"{a}t Berlin, \texttt{mbender@math.tu-berlin.de}} \qquad Simon Telen\thanks{Max Planck Institute for Mathematics in the Sciences, Leipzig, \texttt{simon.telen@mis.mpg.de}}}
\date{}
\maketitle

\vspace{-3.5\baselineskip}
\begin{abstract}
In latest years, several advancements have been made in symbolic-numerical
eigenvalue techniques for solving polynomial systems. In this article, we add to
this list. We design an algorithm which solves systems with isolated solutions reliably and efficiently. In overdetermined cases, it reduces the task to an eigenvalue
problem in a simpler and considerably faster way than in previous methods, and it can
outperform the homotopy continuation approach. We provide many examples and
an implementation in the proof-of-concept Julia package \texttt{EigenvalueSolver.jl}.
\end{abstract}

\keywords{\small polynomial systems, eigenvalue theorem, symbolic-numerical algorithm}

\classification{\small
65H04, 
65H10 
}

\section{Introduction}

Polynomial systems arise in many areas of applied science \cite{sommese_numerical_2005,cox2020applications}.
This paper is concerned with solving such systems of equations using numerical computations, that is, using finite precision, floating point arithmetic.
Two important classes of numerical algorithms are algebraic algorithms \cite{emiris1999matrices,kreuzer_computational_2000} and homotopy continuation methods \cite{burgisser_condition_2013,sommese_numerical_2005}. See \cite[Ch.~2]{cox2020applications} for an overview.
In this work, we focus on algorithms of the former type.

Algebraic algorithms are also called eigenvalue algorithms. They consist of two steps. Step (A) uses linear algebra operations to reduce the problem to an eigenvalue problem or univariate polynomial root finding problem.  Step (B) is to solve the eigenvalue or univariate root finding problem using numerical tools. 
Classical examples include Gr\"obner basis and resultant algorithms, see \cite[Ch.~2]{cox2013ideals} or \cite{bender_thesis_2019}. These use symbolic manipulations for step (A), pushing the numerical linear algebra back to the eigenvalue computation in step (B). The reason for this is that, when performed in finite precision arithmetic, these approaches are numerically unstable for step (A), see for instance \cite{kondratyev2004numerical}. Border basis methods have been developed to remedy this unstable behaviour \cite{mourrain_new_1999,stetter2004numerical} and variants based on nullspace computations were introduced in \cite{dreesen2012back}. Methods for performing step (B) are based on linear algebra \cite{corless1997reordered} or, recently, on multilinear algebra \cite{vanderstukken2017systems}.

Two special types of structured matrices play a central role in
algebraic algorithms: Macaulay (or Sylvester) matrices and
multiplication matrices. Macaulay matrices have a sparse,
quasi-Toeplitz structure. They contain the coefficients of the
equations and are manipulated in step (A).
The result of these manipulations is a set of multiplication
matrices. These are structured in the sense that they
commute. Multiplication matrices represent multiplication operators in
the coordinate ring of the solution set \cite[Ch.~5]{cox2013ideals}
and their eigenstructure reveals the coordinates of the solutions
\cite[Ch.~2]{cox_using_2005}. As the Macaulay matrices are typically much larger than multiplication matrices, step (A) determines the running time of the algorithm. This motivates the efforts in active research, including the present paper, to design algorithms which use smaller Macaulay matrices.

In practice, to construct multiplication matrices from Macaulay matrices, we need to choose a basis for the aforementioned coordinate ring. The numerical stability of step (A) strongly depends on this choice \cite{telen2018stabilized}.
Gr\"obner and border basis methods use bases corresponding to special
sets of monomials. For instance, they require these monomials to come
from a monomial ordering \cite[Ch.~2,\S 2]{cox2013ideals} or to be
`connected-to-1' \cite{mourrain_new_1999}. Recent developments showed
that numerical linear algebra heuristics can be applied to choose
bases that improve the accuracy substantially
\cite{telen2018stabilized}. This has lead to the development of
truncated normal forms \cite{telen2018solving}, which use much more
general bases of monomials coming from QR factorizations with optimal
column pivoting, or non-monomial bases coming from singular value
decompositions or Chebyshev representations~\cite{mourrain2019truncated}.

Other than making a good choice of basis, in order to stabilize
algebraic algorithms it is necessary to take solutions at infinity
into account. Loosely speaking, a polynomial system has solutions at
infinity if the slightest random perturbation of the nonzero
coefficients introduces new solutions with large coordinates. This is
best understood in the language of toric geometry
\cite{cox_toric_2011}. Situations in which there are finitely many
solutions at infinity (see Assumption \ref{ass:zerodim}) can be
handled by introducing an extra randomization in the algorithm, which
was first used in \cite{telen2019numerical,bender_toric_2020}. Where
classically the multiplication matrices represent `multiplication
with a polynomial $g$', the multiplication matrices in these papers
represent `multiplication with a rational function $g/f_0$', where
$f_0$ is a random polynomial that does not vanish at any of the solutions to
the system. For details and a geometric interpretation, we refer to
\cite{telen2019numerical,bender_toric_2020}. We will use a similar approach in this paper. In terms of our results, choosing the denominator $f_0$ randomly is essential in cases where the
conditions in Lemma \ref{thm:propertiesOFh} are not satisfied for
$f_0 = 1$, while they are for a generic $f_0$ (Example \ref{ex:nf0}).

We summarize the contributions of the present paper. First, we adapt
the eigenvalue theorem \cite[Ch.~2, Thm.~4.5]{cox_using_2005} to
reduce the problem of solving polynomial systems to the computation of
eigenvalues.
Our new version allows to compute solutions from matrices that need
not represent classical multiplication operators; see
\Cref{ex:notMultMap}.
We propose an easy-to-state and easy-to-verify criterion for Macaulay
matrices to be `large enough' for constructing such matrices
(\Cref{thm:propertiesOFh}).
Moreover, we identify a broad class of overdetermined polynomial
systems, namely semi-regular unmixed systems, for which these Macaulay
matrices are much smaller than those in classical algorithms,
e.g. \cite{emiris_complexity_1996,massri2016solving}.
We distil these new insights, together with the recent advances in
numerical eigenvalue algorithms explained above, into an algorithm
(Algorithm \ref{alg:solve}).
We introduce the notion of admissible tuples (\Cref{def:condSolv}),
which parametrize Macaulay matrices satisfying our criterion from
\Cref{thm:propertiesOFh} and show how to construct such tuples for
structured systems of equations.
Additionally, we adapt \cite[Sec.~4]{mourrain2019truncated} to obtain
an algorithm for computing smaller admissible tuples for
overdetermined, unmixed systems (\Cref{alg:getATunmixed}).
We provide a Julia implementation of our algorithms, available online
at \link. Our experiments show the efficiency and accuracy of
this package. They contain a comparison with the state-of-the-art  Julia package \texttt{HomotopyContinuation.jl}
\cite{breiding2018homotopycontinuation}. We show that our eigenvalue
methods are competitive, and in strongly overdetermined cases, they
are considerably faster.

To make the paper accessible to a wide audience, we state most of our results and proofs using only terminology from linear algebra. For results that require more background in algebraic (and in particular toric) geometry, we sketch proofs and provide full references.

The paper is organized as follows. In Section \ref{sec:alg}, we
introduce our adapted eigenvalue theorem, admissible tuples and our
algorithm.  In Section \ref{sec:constr}, we present constructions for
admissible tuples for different families of polynomial systems.
Finally, in Section \ref{sec:numexp}, we demonstrate the effectiveness
of our algorithms through extensive numerical experimentation. Our
computations are done using the Julia package
\texttt{EigenvalueSolver.jl}.

\section{The algorithm} \label{sec:alg}

  In this section, we present a symbolic-numerical algorithm to solve
  polynomial systems (Algorithm \ref{alg:solve}). We show that the solutions of the system can be
  obtained from the eigenvalues of certain matrices $M_g$ defined in
  \eqref{eq:defMg}. For some choice of input for Algorithm \ref{alg:solve}, these matrices represent
  \emph{multiplication operators}, see Remark \ref{rmk:multmtcs}. In this case, the results of this section are well-known, e.g. \cite{lazard_resolution_1981}. However, in general, our matrices $M_g$ may not have this interpretation. This is illustrated in Example \ref{ex:notMultMap}. The upshot in these cases
  is that they can be computed more efficiently.

Consider the polynomial ring $R := \C[x_1,\dots,x_n]$ and a tuple of
$s$ polynomials $\f := (f_1,\dots,f_s) \in R^s$, with $s \geq n$. Our aim in this section is to present an algorithm for solving the system of equations $\F(\x) = 0$, where we use the short notation $\x$ for $(x_1,\ldots, x_n)$. A point $\z \in \C^n$ is called a \emph{solution} of $\F$ if $\F(\z) = 0$, that is, $f_i(\z) = 0$ for every $i \in \{1,\dots, s\}$.
For a vector $\alpha = (\alpha_1, \ldots, \alpha_n) \in \N^n$, we denote by $\x^\alpha$ the
monomial $\prod_{i = 1}^n x_i^{\alpha_i} \in R$. We say that $\alpha$
is the \emph{exponent} of the monomial $\x^\alpha$.
In what follows, we write each polynomial $f_i$ as
\[
f_i := \sum_{\alpha \in \N^n} c_{i,\alpha} \, \x^\alpha.
\]
where $c_{i,\alpha} \in \C$ are the \emph{coefficients} of $f_i$ and finitely many of them are nonzero.
We define the \emph{support} $A_i$ of $f_i$ as the set of
exponents $\alpha \in \N^n$ corresponding to non-zero coefficients
$c_{i,\alpha} \in \C$,
\[
  A_i := \{\alpha \in \N^n : c_{i,\alpha} \neq 0 \}.
\]
Given two subsets $E_1,E_2 \subset \N^n$, we denote by $E_1 + E_2$ the
\emph{Minkowski sum} of $E_1, E_2$, that is,
\[
E_1 + E_2 := \{\alpha + \beta : \alpha \in E_1, \beta \in E_2\}.
\]
For a finite set of exponents $E \subset \N^n$, we write $R_{E}$
for the subvector space of $R$ spanned by the monomials with exponent in
$E$. That is, 
\[
R_E := \bigoplus_{\alpha \in E} \, \C \, \cdot \x^\alpha.
\]
Observe that, given $g_1 \in R_{E_1}$ and
$g_2 \in R_{E_2}$, we have that $g_1 \, g_2 \in R_{E_1 + E_2}$.

Consider a tuple of $s$ finite sets of exponents
$\bm{E} := (E_1,\dots,E_s)$, where $E_i \subset \N^n$, and another
finite set of exponents $D \subset \N^n$ such that for every
$i \in \{1, \ldots, s\}$, $D$ contains the exponents in $A_i + E_i$.
An essential ingredient for our eigenvalue algorithm is the \emph{Sylvester map}
\[
\begin{array}{c c c c}
\Sylv_{(\f,\bm{E};D)} : & R_{E_1} \times \dots \times R_{E_s} &
\rightarrow & R_{D}\\ &
(g_1,\dots,g_s) & \mapsto & \sum_i g_i \, f_i.
\end{array}
\]
This is a linear map between finite
dimensional vector spaces, so we can represent it by a matrix
\[\Mat(\f,\bm{E};D) \in \C^{\#D \times (\sum_i \# E_i)}.\]
Matrices obtained by using the standard monomial bases for the vector spaces $R_{E_i}$ and $R_D$ in this representation are often called \emph{Macaulay matrices}.  We index the rows of the matrix
with the exponents belonging to $D$ and the columns with pairs
$\{(i,\beta_i) : i \in \{1, \ldots, s\}, \beta_i \in E_i\}$. The
$(\alpha, (i,\beta_i))$-entry of $\Mat(\f,\bm{E};D)$ contains the
coefficient $c_{i,(\alpha-\beta_i)}$ of $f_i$, that is,
\[
\Mat(\f,\bm{E};D)_{(\alpha, (i,\beta_i))} := c_{i,\alpha-\beta_i}.
\]
Observe that this coefficient might be zero. The ordering of the
exponents is of no importance in the scope of this work. We will
therefore not specify it and assume that some ordering is fixed for
all tuples $A_i, E_i, D$ throughout the paper.

\begin{example}
  To avoid subscripts, we replace the variables $x_1$ and $x_2$
  by $x$ and $y$, respectively. Consider the sets of
  exponents $A_1,\dots,A_3$ and the system $\f := (f_1, f_2, f_3)$ given by
  \[
  \begin{array}{c}
    A_1 := \{(0,0),(1,0),(0,1),(0,2)\}, \\
    A_2 := \{(0,0),(1,0),(2,0),(0,1)\},\\
    A_3 := \{(0,0),(1,0),(2,0),(0,1)\},
  \end{array}
  \qquad
    \begin{array}{l l}
  f_1 := -1  + 2 \, x + 2 \, y \phantom{{}^2} + y^2 &  \in R_{A_1},\\ 
  f_2 := -1 + \phantom{2 \,} x + \phantom{2 \,} x^2 + y & \in R_{A_2} ,\\ 
  f_3 := -1 + 2 \, x + 2 \, x^2 + y  & \in R_{A_3}. 
  \end{array}
  \]
  We construct the Macaulay matrix $\Mat(\f,\bm{E};D)$, where $\bm{E} := (E_1,E_2,E_3)$ and
  \[
  \begin{array}{l}
    E_1 := \{(0,0),(1,0)\}, \\
    E_2 = E_3 := \{(0,0),(0,1)\},
  \end{array}
  \qquad
  \begin{array}{c c}
    D := & \{(0,0),(1,0),(2,0),(0,1),\\
         & (1,1),(2,1),(0,2),(1,2)\}.
  \end{array}
  \]
  \[
  \Mat(\f,\bm{E};D) =
  \begin{array}{ c | r r r r r r }
    	& f_1 & x\,f_1 & f_2 & y \, f_2 & f_3 & y\,f_3 \\ \hline
    1	&-1& &-1& &-1& \\ 
    y	&2& &1&-1&1&-1\\ 
    y^2	&1& & &1& &1\\ 
    x	&2&-1&1& &2& \\ 
 x \, y	& &2& &1& &2\\ 
 x\,y^2	& &1& & & & \\ 
    x^2	& &2&1& &2& \\ 
 x^2\,y	& & & &1& &2
  \end{array} 
  \]
\end{example}
\begin{remark}
  \label{rmk:leftMult} 
  Given $\zeta \in \C^n$ and a finite subset $E \subset \N^n$, we denote by $\zeta^E$ the row vector
  \[
  (\zeta^\alpha : \alpha \in E).
  \]
  The vector obtained by the product
  $\zeta^D \cdot \Mat(\f,\bm{E};D) \in \C^{\sum_i \#E_i}$ is indexed by the tuples
  $\{(i,\beta_i) : i \in \{1, \ldots, s\}, \beta_i \in E_i\}$ and the $(i,\beta_i)$-entry is given by $f_i(\zeta) \, \zeta^{\beta_i}$. If $\zeta$ is a solution of
  $\f$, then $\zeta^D$ belongs to the cokernel of
  $\Mat(\f,\bm{E};D)$. Moreover, if $\z \in \C^n$ is such that $\z^{E_i} \neq 0$ for all $i$, the opposite implication also holds. This is the case, for instance, for any solution $\z \in (\C \setminus \{0\})^n$. 
\end{remark}

We define the value
$\HF(\f,\bm{E};D)$ as the corank
of $\Mat(\f,\bm{E};D)$:
\[
\HF(\f,\bm{E};D) := \#D - \textup{Rank}(\Mat(\f,\bm{E};D)).
\]
Let
$\Cok(\f,\bm{E};D) \in \C^{\HF(\f,\bm{E};D) \times \#D}$ be a cokernel matrix (or left null space matrix) of $\Mat(\f,\bm{E};D)$. That is, $\Cok(\f,\bm{E};D)$ has rank $\HF(\f,\bm{E};D)$ and
\[\Cok(\f,\bm{E};D) \cdot \Mat(\f,\bm{E};D) = 0\]
We will index the columns of $\Cok(\f,\bm{E};D)$ with the
exponents in $D$.

\begin{example}[Cont.]
  The system $\f$ has one solution $(-1,1) \in \C^2$. The vector
  \[(-1,1)^D = (1, 1, 1, -1, -1, -1, 1, 1)\] belongs to the cokernel of
  $\Mat(\f,\bm{E};D)$. Moreover, we have that $\HF(\f,\bm{E};D) = 2$
  and
      \[ \belowdisplayskip=-\baselineskip
   \Cok(\f,\bm{E};D) =     
    \begin{blockarray}{ r r r r r r r r}
      1 & y & y^2 & x & x \, y & x \, y^2 & x^2 & x^2 \, y \\
        \begin{block}{[r r r r r r r r]}
    1 & 1 & 1 & -1 & -1 & -1 & 1 & 1 \\
      0 & 0 & 0 & 0 & -1 & 2 & 0 & 1    \\
\end{block}
\end{blockarray}. \qedhere \]
\end{example}

Consider two finite sets of exponents $A_0, E_0 \in \N^n$ such that $A_0 + E_0 \subset D$.
For each polynomial $f_0 \in R_{A_0}$, we define the matrix $N_{f_0}$ as
\begin{align}\label{eq:nf0}
N_{f_0} :=  \Cok(\f,\bm{E};D) \cdot \Mat(f_0,E_0;D).
\end{align}
Observe that $N_{f_0} \in \C^{  \HF(\f,\bm{E};D) \times \#E_0}$ and the columns of
$N_{f_0}$ are indexed by the exponents in $E_0$ (more precisely, by
the pairs $(0,\alpha)$ for each $\alpha \in E_0$).

\begin{lemma}
  \label{thm:propertiesOFh}
    For any $f_0 \in R_{A_0}$ we have that
    \[
      \HF((f_0,\f),(E_0,\bm{E});D) = 0 ~ \Longleftrightarrow ~   N_{f_0} ~\text{ has rank }~ \HF(\f,\bm{E};D),
    \]
    where $(f_0,\f) = (f_0,f_1,\dots,f_s)$ and
    $(E_0,\bm{E}) = (E_0,E_1,\dots,E_s)$.
    Moreover, in that case, for every solution $\zeta \in \C^n$ of
    $\f$ such that the vector $\zeta^{D}$ is non-zero,
    we have $f_0(\zeta) \neq 0$.
  \end{lemma}

  \begin{proof}
    The $\Rightarrow$ direction of the first statement follows
    directly from
    \[
    \left(\begin{array}{cc}
      N_{f_0} & 0
    \end{array}
              \right)    =  \Cok(\f,\bm{E};D) \cdot
    \Mat((f_0,\f),(E_0,\bm{E});D).
              \]
     For the $\Leftarrow$ direction, suppose that $N_{f_0}$ has rank $\HF(\f,\bm{E};D)$. Then $\Cok(f_0,E_0;D) \cap \Cok(\f,\bm{E};D) = \{0\}$, which implies that the cokernel of $\Mat((f_0,\f),(E_0,\bm{E});D)$ is trivial, and hence it has rank $\#D$.
              The second statement follows from the fact that
              $\Mat((f_0,\f),(E_0,\bm{E});D)$ has trivial cokernel;
              as $\zeta^{D}$ is a non-zero vector, if
              $f_0(\zeta) = 0$, by \Cref{rmk:leftMult}, $\zeta^{D}$
              belongs to the cokernel of
              $\Mat((f_0,\f),(E_0,\bm{E});D)$.
  \end{proof}

  \begin{example}[Cont.] \label{ex:nf0}
    We consider the sets of exponents $A_0,E_0$ and the polynomial
    $f_0$ given by 
    \[
      A_0 = \{(0,0),(1,0),(0,1)\}, \qquad E_0 = \{(0,0),(1,0),(0,1)\}, \qquad f_0 := 1 + 3 \, x + y  \in R_{A_0}.
      \]
      In this case, we have 
     \[
    N_{f_0} =
    \begin{blockarray}{ccc}
      1 & x & y \\
      \begin{block}{[rrr]}
        -1 & 1 & -1 \\
        0 & -1 & -3 \\
      \end{block}
    \end{blockarray}.\]
  Observe that, even though $1 \in R_{A_0}$, the matrix
  $N_{1}$ is not full-rank.
\end{example}
  
In what follows, we say that a property holds for \emph{generic}
points of a vector space if it holds for all points not contained in a
subset of Lebesgue measure zero.  Note that $N_{f_0}$ is a matrix
whose entries depend linearly on the coefficients of $f_0$. This means
that if there exists $f_0 \in R_{A_0} $ such that $N_{f_0}$ has rank
$\HF(\f,\bm{E};D)$, then $\rank (N_h) = \HF(\f,\bm{E};D)$ for generic
elements $h \in R_{A_0}$.  Below, we assume that there is
$f_0 \in R_{A_0}$ such that $\rank(N_{f_0}) =\HF(\f,\bm{E};D)$ and we
fix such an $f_0 \in R_{A_0}$. This assumption is very mild and given
$\F, A_0, {\bm E}, D$, it is easy to check if it holds.
  
For ease of notation, we will write $\gamma := \HF(\f,\bm{E};D)$.
Given a set of exponents $\B \subset E_0$, we define the submatrix
$N_{f_0,\B} = \Cok(\f,\bm{E};D) \cdot \Mat(f_0,\B;D) \in \C^{\gamma \times \#{\B}}$ of $N_{f_0}$ consisting of
its columns indexed by $\B$.
We fix $\B \subset E_0$ of cardinality $\gamma$ such that
$N_{f_0,\B} \in \C^{\gamma \times \gamma}$ is invertible.
For each $g \in R_{A_0}$, we define the matrix
$M_g \in \C^{\gamma \times \gamma}$, defined as
\begin{align}
  \label{eq:defMg}
  M_g := N_{g,\B} \cdot N_{f_0,\B}^{-1} .
\end{align}

  \begin{example}
    [Cont.] \label{ex:basisB}

    We fix the basis $\B = \{1,x\}$ and the matrix
    $N_{f_0,\B} = \left[\begin{smallmatrix} -1 & 1 \\ 0 &
        -1 \end{smallmatrix}\right]$.
    Then, for $g = -1 + 3 \, x + 2 \, y$, we have
  \[\belowdisplayskip=-\baselineskip
  M_g = \left[\begin{array}{rr} -2&2\\0&-2\end{array}\right] \cdot
  \left[\begin{array}{rr} -1&-1\\0&-1\end{array}\right] =
  \left[\begin{array}{rr} 2&0\\0&2\end{array}\right]
  \qquad
  \text{and}
  \qquad
  M_x =
  \left[\begin{array}{rr} 1&0\\0&0\end{array}\right]
  \]
\end{example}
  
  \begin{remark}
    \label{rmk:linearity}
    The map $g \in R_{A_0} \mapsto M_g \in \C^{\gamma \times \gamma}$
    is a linear map. That is, for $\lambda \in \C$, $g_1,g_2 \in R_{A_0}$,
    \[
    M_{g_1 + \lambda \, g_2} =    M_{g_1} + \lambda \, M_{g_2}.
    \]
    Moreover, $M_{f_0}$ is the identity matrix.
  \end{remark}
  A key observation is that we can solve the system of equations
  $\f(\x) = 0$ by computing the eigenstructure of these matrices
  $M_g$, for $g \in R_{A_0}$.
  For that, we adapt the classical eigenvalue theorem from
  computational algebraic geometry (see Remark \ref{rmk:multmtcs}).
  We say that a non-zero row vector $v$ is a \emph{left eigenvector}
  of a matrix $M$ with corresponding eigenvalue $\lambda$ if it
  satisfies $v \cdot M = \lambda v$.
  \begin{theorem}[Eigenvalue theorem]
    \label{thm:evthm}
    Using the notation introduced above, consider a polynomial
    system $\f$ and a polynomial $f_0$ such that $N_{f_0}$ has full-rank;
    see \Cref{eq:nf0}.
    For each solution $\zeta \in \C^n$ of $\f$ such that
    $\zeta^{D} \neq 0$, $M_g$ from \eqref{eq:defMg} has a left eigenvector $v_\z$ such that
    $v_\z \cdot \Cok(\f,\bm{E};D) = \zeta^{D}$. The corresponding
    eigenvalue is $\frac{g}{f_0}(\zeta)$.
    Conversely, if $v$ is a left eigenvector of $M_g$ such that
    $v \cdot \Cok(\f,\bm{E};D) $ is proportional to $\z^D \neq 0$ for
    some $\z \in \C^n$ such that $\z^{E_i} \neq 0$ for all $i$, then
    $\z$ is a solution of $\f$. Moreover, the corresponding eigenvalue
    of $M_g$ is $\frac{g}{f_0}(\zeta)$.
  \end{theorem}

  \begin{remark} \label{rmk:multmtcs}
    In some cases, the previous theorem can be derived from the classical \emph{eigenvalue theorem} from computational algebraic geometry, where $M_g$
    represents
    the \emph{multiplication map}
    \[M_g : R / \ideal{\f} \rightarrow R / \ideal{\f}, h \mapsto M_g(h)
    = h \, g, \quad \text{for } g \in R. \]
    Here $\ideal{F} \subset R$ is the ideal generated by $f_1, \ldots, f_s$. As we will see (Example \ref{ex:notMultMap}), this is not always the case. In the context of computer algebra, the eigenvalue theorem was introduced in
    \cite{lazard_resolution_1981} (eigenvalues) and in
    \cite{auzinger_elimination_1988} (eigenvectors). For a historic overview and a proof in terms of matrices, see \cite{cox2020stickelberger} and \cite{emiris_complexity_1996} respectively.
  \end{remark}
  
  To prove Theorem \ref{thm:evthm}, we need two auxiliary lemmas.

  \begin{lemma}
    \label{thm:hasSolNgNonFullRank}
    Let $\zeta \in \C^n$ be a solution of $\f$ such that $\zeta^{D} \neq 0$ and let $g \in R_{A_0}$ be such that $g(\z) = 0$. Then, the matrix $M_g$ is singular. 
  \end{lemma}

  \begin{proof}
    By \Cref{rmk:leftMult}, $\zeta^{D}$ belongs to the cokernel of
    $\Mat(\f,\bm{E};D)$, hence there is a row vector $v_\z \in \C^{\gamma} \setminus \{0\}$ such that
    $v_\z \cdot \Cok(\f,\bm{E};D) = \zeta^{D}$. Moreover, $\zeta^{D}$
    belongs to the cokernel of $\Mat(g,E_0;D)$. Hence,
    $v_\z \cdot N_{g} = 0$ and so $v_\z \cdot N_{g, \B} = 0$.
  \end{proof}
  
  \begin{lemma} \label{lem:nonzB}
  Let $\z \in \C^n$ be a solution of $\f$. If $\z^\B = 0$, then $\z^D = 0$. 
  \end{lemma}
\begin{proof}
  Since $\z$ is a solution of $\f$, there is a row vector $v_\z$ such that
  $v_\z \cdot \Cok(\f,\bm{E};D) = \z^D$. By observing that
  $N_{f_0,\B} = \Cok(\f,\bm{E};D) \cdot \Mat(f_0,\B;D)$, the lemma follows from
\begin{align*}
v_\z \cdot N_{f_0,\B} &= (v_\z \cdot \Cok(\f,\bm{E};D)) \cdot \Mat(f_0,\B;D) \\
&= \z^D \cdot  \Mat(f_0,\B;D) \\
&= (\z^\alpha f_0(\z) : \alpha \in \B) = 0,
\end{align*}
where the last line uses $\z^\B = 0$. Since $N_{f_0,\B}$ is invertible, this implies $v_\z = 0$, and thus $\z^D = 0$. 
\end{proof}

  \begin{proof}[Proof of \Cref{thm:evthm}]
    The proof is based on the following observations.
    By Remark~\ref{rmk:linearity}, the eigenvalues of $M_{g}$
    correspond to the values $\lambda \in \C$ such that
    $M_{g} - \lambda \, \id = M_{g - \lambda \, f_0}$ is singular.
    By \Cref{thm:propertiesOFh}, we have $f_0(\zeta) \neq 0$ for each
    solution $\z$ of $\f$.
    Let $\z$ be a solution of $\f$. If
    $\lambda = \frac{g}{f_0}(\zeta)$, the polynomial
    $g - \lambda \, f_0 \in R_{A_0}$ vanishes at $\z$. As by
    assumption $\zeta^{D} \neq 0$, from \Cref{thm:hasSolNgNonFullRank}
    we deduce that $M_{g - \lambda \, f_0}$ is singular. Therefore,
    $\frac{g}{f_0}(\zeta)$ is an eigenvalue of $M_g$. For the
    associated left eigenvector, let $v_\z$ be as in the proof of
    \Cref{thm:hasSolNgNonFullRank}. We have
    $v_\z \cdot N_{g-\lambda f_0,\B} = v_\z \cdot (N_{g,\B} - \lambda
    N_{f_0,\B}) = 0$ and multiplying from the right by
    $N_{f_0,\B}^{-1}$ gives $v_\z \cdot (M_g - \lambda \id) = 0$.
    
    Conversely, suppose that $v$ is a left eigenvector of $M_g$ such
    that $v \cdot \Cok(\f,\bm{E};D) = \z^{D} \neq 0$ for some
    $\z \in \C^n$ (we may assume equality after scaling). By
    \Cref{rmk:leftMult}, under the assumption $\z^{E_i} \neq 0$ for
    all $i$, $\z$ is a solution of $\f$, see Remark \ref{rmk:leftMult}. We now compute the
    corresponding eigenvalue. By definition,
    $v \cdot (M_g - \lambda \id) = 0$ for some $\lambda$. Multiplying
    from the right by $N_{f_0,\B}$ we see that
    $v \cdot N_{g-\lambda f_0,\B} = (v \cdot \Cok(\f,\bm{E};D) ) \cdot
    \Mat( g-\lambda f_0, \B, D) = \z^D \cdot \Mat( g-\lambda f_0, \B,
    D) = 0$. By \Cref{lem:nonzB}, $\z^\B \neq 0$ and since
    $f_0(\z) \neq 0$ (\Cref{thm:propertiesOFh}) we conclude
    $\lambda = \frac{g}{f_0}(\zeta)$.
  \end{proof}

  \begin{example}
    [Cont.] \label{ex:notMultMap}

    The unique eigenvalue of
    $M_g = \left[\begin{smallmatrix} 2 & 0 \\ 0 &
        2 \end{smallmatrix}\right]$ is
    $2 = \frac{g}{f_0}(-1,1)$. Moreover, we have that
    $1 = \frac{x}{f_0}(-1,1)$ is an eigenvalue of $M_x = \left[\begin{smallmatrix} 1 & 0 \\
        0 & 0 \end{smallmatrix}\right]$ whose associated eigenvector
    $(1,0)$ satisfies $(1,0) \cdot \Cok(\f,\bm{E};D) = (-1,1)^D$.
    Observe that, as the system $\f$ has only one solution, namely
    $(-1,1)$, with multiplicity one, the matrices $M_x$ and $M_g$ are
    not multiplication operators.
  \end{example}
  
  We now characterize row vectors $v$ that are an eigenvector of all
  matrices in $\M := \{M_h ~:~ h \in R_{A_0} \}$. Observe that, by
  Remark~\ref{rmk:linearity}, $\M$ is a vector space.
  For any nonzero row vector $v \in \C^{\gamma} \setminus \{0\}$, we
  define the subspace
  \[
    \M(v) = \{M_h \in \M ~:~ v \text{ is a left eigenvector of } M_h \}.
  \]
  One can check that $\M(v)$ is a vector space. We say that $v$ is an
  eigenvector of $\M$ if $\M(v) = \M$.

\begin{example}
  [Cont.] \label{ex:eigenvectors}

  Any vector in $\C^2$ is an eigenvector of
  $M_g = \left[\begin{smallmatrix} 2 & 0 \\ 0 &
      2 \end{smallmatrix}\right]$, but for $h = 1+x+y$, we obtain
  $M_{h} = \left[\begin{smallmatrix} -1 & 0 \\ 0 &
      1 \end{smallmatrix}\right]$, which has left
  eigenvectors $(1,0)$ and $(0,1)$. These vectors are also
  eigenvectors of
  $M_x = \left[\begin{smallmatrix} 1 & 0 \\ 0 &
      0 \end{smallmatrix}\right]$. It is straightforward to check that
  $\{M_g,M_{h},M_x\}$ generate $\M$ as a vector space, so that
  $(1,0)$ and $(0,1)$ are eigenvectors of $\M$.
  This means that there might be eigenvectors of $\cal M$ which are
  not associated to solutions of $\f$.
  Also, observe that the matrices in $\cal M$ commute.
\end{example}

\begin{proposition} \label{prop:commonevec}
  Fix a non-zero vector $v \in \C^\gamma$. We have that $v$ is an
  eigenvector of $\M$ if and only if it is a left eigenvector of
  $M_h$ for a generic $h \in R_{A_0}$.
\end{proposition}
\begin{proof}
The `only if' direction is clear. For the `if' direction, let $U \subsetneq R_{A_0}$ be the subset of polynomials $h$ for which $v$ is an eigenvector of $M_{h}$. It is easy to see that $U$ is a linear subspace, so that it is closed in the Zariski topology. It is also dense by assumption, so we conclude that $U = R_{A_0}$. 
\end{proof}

\begin{proposition} \label{prop:multeigenval}
  Consider eigenvectors $v_1, \ldots, v_m$ of $\M$. We have that
  $v_1, \ldots, v_m$ correspond to the same eigenvalue of $M_h$ for
  all $M_h \in \M$ if and only if they correspond to the same
  eigenvalue of $M_h$ for a generic $h \in R_{A_0}$.
\end{proposition}
\begin{proof}
Again, the `only if' direction is clear. For the opposite implication, note that 
\[ W =  \{ h \in R_{A_0} ~:~ \text{$v_1, \ldots, v_m$ correspond to the same eigenvalue of $M_h$} \} \]
is a vector subspace of $R_{A_0}$. Since it is also dense by hypothesis, we conclude that $W = R_{A_0}$ as in the proof of Proposition \ref{prop:commonevec}.
\end{proof}

\begin{example}
  [Cont.]

  Instead of the basis $B$ fixed in
  \Cref{ex:basisB}, in what follows we consider $B = \{x,y\}$ and so
  $N_{f_0,B'} = \left[\begin{smallmatrix} 1 & - 1 \\ - 1 &
      3 \end{smallmatrix}\right]$. This way, we obtain the following
  matrices:

    \[\belowdisplayskip=-\baselineskip
      M'_g = N_{g,B'} \, N_{f_0,B'}^{-1} =
      \left[\begin{array}{rr} 2&0\\-\frac{3}{4}&\frac{5}{4}\end{array}\right]
      \qquad
      \text{and}
      \qquad
      M'_x =
      \left[\begin{array}{rr} 1&0\\\frac{1}{4}&\frac{1}{4}\end{array}\right]
    \]

    While $(1,0)$ is a common left eigenvector of $M'_g$ and
    $M'_x$ corresponding to the unique solution of the system, the
    vector $(1,1)$ is a left eigenvector of $M'_g$, but it is not of
    $M'_x$. Comparing this example with \Cref{ex:eigenvectors}, we
    observe that, depending on the choice of the basis $B$, there may be spurious common eigenvectors of $\cal M$.
  \end{example}

Propositions \ref{prop:commonevec} and \ref{prop:multeigenval} give a simple procedure for computing, for a given eigenvalue $\lambda_g$ of $M_g$, the intersection of the corresponding left eigenspace of $M_g$ with the eigenvectors of $\M$. Suppose that this eigenspace is spanned by the rows of the matrix $V_{\lambda_g}$. By \Cref{prop:commonevec}, we simply need to check which elements in the row span of $V_{\lambda_g}$ are also eigenvectors of $M_{h}$, for a random element $h \in R_{A_0}$. Proposition \ref{prop:multeigenval} guarantees that these eigenvectors, if they exist, belong to a unique eigenvalue of $M_{h}$. This is summarized in Algorithm \ref{alg:geteigenspace}. 
\begin{algorithm}[h]
    \caption{\textsc{GetEigenspace}}
    \label{alg:geteigenspace}
    \begin{algorithmic}[1]
    \small
    	\Require{An eigenvalue $\lambda_g$ of $M_g$ for generic $g\in R_{A_0}$, a matrix $V_{\lambda_g}$of size $m \times \gamma$ whose rows contain a basis for the corresponding eigenspace and the matrix $M_{h}$ for a generic $h \in R_{A_0}$}
    	 \Ensure{A matrix ${\cal V}$ whose rows are a basis for the intersection of the row span of $V_{\lambda_g}$ with the eigenvectors of $\M$.}
     \If{$m = 1$}
     \If{$\rank \begin{bmatrix} V_{\lambda_g} \\ V_{\lambda_g} \cdot M_{h} \end{bmatrix} = 1$}
     \State ${\cal V} \gets V_{\lambda_g}$
     \Else
     \State ${\cal V} \gets \{0\}$
     \EndIf
     \Else
     \State $O \gets$ random matrix of size $\gamma \times m$
     \State $\{ (\mu_i, C_i) \} \gets$ solve the GEP 
     $c_i \cdot (V_{\lambda_g}
     \cdot M_{h} \cdot O) = \mu_i c_i \cdot (V_{\lambda_g} \cdot
     O)$ \label{line:GEP}
     \State $C \gets$ $C_i$ such that $C_i \cdot V_{\lambda_g}
     \cdot M_{h}  =\mu_i \cdot C_i \cdot V_{\lambda_g}$
     \label{line:C}
     \If{$C$ is empty}
     \State ${\cal V} \gets \{0\}$
     \Else
     \State ${\cal V} \gets C \cdot V_{\lambda_g}$
     \EndIf
          \EndIf
          \State \textbf{return} ${\cal V}$
	\end{algorithmic}
      \end{algorithm}
In line \ref{line:GEP} of the algorithm, we solve the generalized eigenvalue problem (GEP) given by the pencil $(B_1,B_2) := (V_{\lambda_g} \cdot M_{h} \cdot O, V_{\lambda_g} \cdot O)$, that is, we compute all eigenvalues $\mu_i$ and a basis for the left eigenspace $C_i = \{ c_i \in \C^{m}~|~ c_i B_1 = \mu_i c_i B_2 \}$. In line \ref{line:C}, we select (if possible) the unique eigenvalue $\mu_i$ whose corresponding eigenspace $C_i$  gives the desired intersection ${\cal V} = C_i \cdot V_{\lambda_g}$. \Cref{prop:multeigenval} also has the following direct corollary. 
\begin{corollary} \label{cor:uniquetuple}
Let $\lambda_g$ be an eigenvalue of $M_g$ and let ${\cal V} \in \C^{m \times \gamma}$ be a matrix whose rows are a basis for the left eigenspace of $M_g$ corresponding to $\lambda_g$, intersected with the eigenvectors of $\M$.  If $g$ is generic, there is exactly one tuple $(\lambda_\alpha)_{\alpha \in A_0} $ such that
 \[ {\cal V} M_{x^\alpha} = \lambda_\alpha {\cal V}, \quad \text{for all }\alpha \in A_0. \]
\end{corollary}

\begin{remark} \label{rem:computeeval}
This has the practical implication that $\tilde{M}= {\cal V} M_{x^\alpha} T ({\cal V} T)^{-1} = \diag(\lambda_\alpha, \ldots, \lambda_\alpha)$ for a random matrix $T \in \C^{\gamma \times m}$ has only one eigenvalue $\lambda_\alpha$, equal to $\text{trace}(\tilde{M})/m$. If ${\cal V}$ has only one row we obtain $\lambda_\alpha$ from the Rayleigh quotient $\lambda_\alpha = {\cal V} M_{x^\alpha} {\cal V}^*/({\cal V} {\cal V}^*)$, where ${}^*$ is the conjugate transpose.
\end{remark}

\begin{proposition}[Criterion for eigenvalues]
  \label{prop:eigenvalcrit}
  Let $\lambda_g$ be an eigenvalue of $M_g$. If
  $\lambda_g = \frac{g}{f_0}(\z)$ for some solution $\z \in \C^n$ of
  $\f$ satisfying $\z^D \neq 0$, 
  then the tuple
  $(\lambda_\alpha)_{\alpha \in A_0}$ from \Cref{cor:uniquetuple}
  satisfies
  \[  C \cdot (\lambda_\alpha)_{\alpha \in A_0} = \z^{A_0} \text{ for some } C \in \C \setminus \{0\}. \]
\end{proposition}

\begin{proof}
Let
  ${\cal V} \in \C^{m \times \gamma}$ be a matrix whose rows are a
  basis for the left eigenspace of $M_g$ corresponding to $\lambda_g$
  intersected with the eigenvectors of $\M$.
If $\lambda_g = \frac{g}{f_0}(\z)$ for some solution $\z \in \C^n$ of $\f$ satisfying $\z^D \neq 0$, then by Theorem \ref{thm:evthm} we know that $v_\z$ is a corresponding eigenvector of $\M$. Therefore, there exists $c_\z \in \C^m \setminus \{0\}$ such that $c_\z {\cal V} = v_\z$. Another consequence of Theorem \ref{thm:evthm} is 
\[ c_\z {\cal V} M_{x^\alpha} = \frac{x^\alpha}{f_0}(\z) c_\z {\cal V}. \qedhere \]
\end{proof}

The results discussed above suggest several ways of extracting the coordinates of a solution $\z \in \C^n$ of $\F(x) = 0$ form the eigenstructure of the matrices $M_g$. Both the eigenvectors (Theorem \ref{thm:evthm}) and the eigenvalues (Proposition \ref{prop:eigenvalcrit}) reveal vectors of the form $\z^A$ for some set of exponents $A \subset \N^n$.
We now recall how to compute the coordinates of $\z$ from the vector $\z^A$ and discuss the assumptions that we need on $A$ in order to be able to do this.

For any subset $A \subset \N^n$, we write
 \[ \N A := \{ \sum_{\alpha \in A} n_\alpha \cdot \alpha ~:~ n_\alpha \in \N \} \subset \N^n, \quad  \Z A := \{ \sum_{\alpha \in A} m_\alpha \cdot  \alpha ~:~ m_\alpha \in \Z \} \subset \Z^n.\]
If $A = \{\alpha_1, \ldots, \alpha_k\}$ with $\alpha_1 = 0$ and the condition $\Z A = \Z^n$ is satisfied, then for $\ell = 1, \ldots, n$, there exist integers $m_{2,\ell}, \ldots, m_{k,\ell}$ such that $m_{2,\ell} \alpha_2 + \cdots + m_{k,\ell} \alpha_k = e_\ell$, where $e_\ell$ is the $\ell$-th standard basis vector of $\Z^n$. These integers $m_{j,\ell}$ can be computed, for instance, using the Smith normal form of an integer matrix whose columns are the elements of $A$. If this is the case, from $\z^A = (\z^{\alpha_1}, \ldots, \z^{\alpha_k})$ we can compute the $\ell$-th coordinate $\z_{\ell}$ of $\z$ as
\begin{equation} \label{eq:coordsfromev}
\z_{\ell} = \prod_{j = 2}^k ( \z^{\alpha_j} )^{m_{j,\ell}}, \quad \ell = 1, \ldots, n.
\end{equation}
This approach can be used to compute the coordinates of $\z \in ( \C \setminus \{0\} )^n$ from $\z^A$, i.e. all points with all non-zero coordinates. Note that some of the $m_{j,\ell}$ may be negative, which may be problematic in the case where $\z$ has zero coordinates. If the stronger condition $\N A_0 = \N^n$ is satisfied (this implies $e_\ell \in A$, $\ell = 1, \ldots, n$), then the integers $m_{j,\ell}$ can be taken non-negative and we can obtain the coordinates of all points $\z$ in $\C^n$ from $\z^A$. We will continue under the assumption that we are mostly interested in computing points in $( \C \setminus \{0\} )^n$, as this is commonly assumed in a sparse setting. However, solutions in $\C^n$ can be computed by replacing $\Z A = \Z^n$ in what follows by the stronger assumption $\N A = \N^n$. Note that if $\Z A = \Z^n$, the outlined approach suggests a way of checking whether or not a vector $q \in \C^{\#A}$ with $q_{\alpha_1}=1$ is of the form $\z^A$ for some $\z \in ( \C \setminus \{0\} )^n$. Indeed, one computes the coordinates $\z_\ell = \prod_{j=2}^k (q_{\alpha_j})^{m_{j,\ell}}$ and checks whether $\z^A = q$.

We turn to the eigenvalue method for extracting the roots $\z$ from  the matrices $M_g$. 
Let $\{ \z_1, \ldots, \z_\D \} \subset \C^n$ be a set of solutions of $\f$ such that $\z_i^D \neq 0$ for all $i$. By Theorem \ref{thm:evthm}, for each of these solutions there is an eigenvalue $\lambda_g$ of the matrix $M_g$ and a space of dimension $m \geq 1$, spanned by the rows of a matrix ${\cal V}$, of eigenvectors of $\M$. Suppose we have computed this matrix ${\cal V}$ (for instance, using Algorithm \ref{alg:geteigenspace}).
We write $A_0 = \{\alpha_1, \ldots, \alpha_k\} \subset \N^n$ and assume that $\alpha_1 = 0$. The unique eigenvalue (Corollary \ref{cor:uniquetuple}) of $M_{x^{\alpha_j}}$ corresponding to ${\cal V}$ is denoted by $\lambda_{ij}$ and can be computed using Remark \ref{rem:computeeval}.
As $\zeta^D \neq 0$, by \Cref{thm:evthm}, there is $v \in \cal V$
such that $v \, \cdot \textsc{Coker}(F, E, D) = C \cdot \zeta_i^D$, for non-zero
$C \in \C$. Therefore, by \Cref{thm:evthm},
\begin{equation} \label{eq:evs}
\lambda_{ij} = \frac{\z_i^{\alpha_j}}{f_0(\z_i)} \quad \text{and hence} \quad \frac{\lambda_{ij}}{\lambda_{i1}} = \z_i^{\alpha_j}, \quad i = 1, \ldots, \D, j = 1, \ldots, k.
\end{equation}
 We would like to recover the coordinates of $\z_i$ from the tuple $(\z_i^{\alpha_2}, \ldots, \z_i^{\alpha_k}) \in \C^{k-1}$. Assuming $\Z A_0 = \Z^n$ and applying \eqref{eq:coordsfromev}, we find 
 \begin{equation}
\z_{i,\ell} = \prod_{j = 2}^k \left( \frac{\lambda_{ij}}{\lambda_{i1}} \right )^{m_{j,\ell}}, \quad \ell = 1, \ldots, n.
\end{equation}
\begin{remark}
In many cases, one can take $A_0 = \{0, e_1, \ldots, e_n \}$, in which case $m_{j,\ell} = 1$ if $j = \ell + 1$ and $m_{j,\ell} = 0$ otherwise.
\end{remark}
Motivated by this discussion, we make the following definition. 
\begin{definition}
  \label{def:condSolv}
  We say that a tuple
  $(\F = (f_1,\dots,f_s),A_0,(E_0,{\bm E}) = (E_0,\dots,E_s),D)$ is
  \emph{admissible} if it satisfies the following three conditions,
    \begin{itemize}
    \item \textbf{Compatibility condition:} For $i = 0, \ldots, s$,
      $A_i + E_i \subset D$.
    \item \textbf{Rank condition:} There exists
      $f_0 \in R_{A_0}$ such that
      $\rank(N_{f_0}) = \HF(\F,{\bm E};D)$.  
    \item \textbf{Lattice condition:} The set $A_0$ satisfies
      $0 \in A_0$ and $\Z A_0 = \Z^n$.
    \end{itemize}
\end{definition}

The results in this section lead to Algorithm \ref{alg:solve} for
solving $\F(x) = 0$, given an admissible tuple
$(\f,A_0,(E_0,\bm{E}),D)$.
This algorithm computes a candidate set of solutions containing
every solution in $(\C \setminus \{0\})^n$.
It might contain spurious points, since there might be eigenvalues that do not correspond to solutions but do come from a common eigenvector of ${\cal M}$, see Example \ref{ex:notMultMap}. One can identify these points, for instance, by evaluating the \emph{relative backward error}, see \Cref{eq:BWE}.
In what follows, we discuss some aspects of the algorithm in more
detail. In practice, the number of columns $\sum_{i = 1}^s \# E_i $ of the
Macaulay matrix $\Mat(\f,\bm{E};D)$ is often much larger than the
number $\# D$ of rows. Multiplying from the right by a random matrix
of size $( \sum_{i = 1}^s \# E_i ) \times \#D$ does not affect the
left nullspace, but reduces the complexity of computing it. This is
what happens in line \ref{line:compress}. See
\cite[Sec.~4.2]{mourrain2019truncated} for details. If
$( \sum_{i = 1}^s \# E_i ) \lesssim \#D$, that is, the number of
columns is not much larger than the number of rows, this step can be
skipped.

\begin{remark}
  By the lattice condition, we have that
  $1 \in R_{A_0}$. However, the rank condition might not be satisfied for $f_0 = 1$. That is, it might happen that
  $\rank(N_{1}) < \Cok(\F,{\bm E};D)$.
 This is the case, for instance
 , in \Cref{ex:nf0}.
  To overcome this issue, we choose $f_0$ randomly in $R_{A_0}$.
\end{remark}

\paragraph{\small Numerical considerations.}
In theory, we may pick $\B$ arbitrary such that $N_{f_0,\B}$ is an
invertible matrix. In practice, \emph{it is crucial to pick $\B$ such
  that $N_{f_0,\B}$ is well-conditioned}. This was shown in
\cite{telen2018solving,telen2018stabilized}. For that, we select a
random $f_0$ and, in line \ref{line:qr}, we use a standard numerical
linear algebra procedure for selecting a well-conditioned submatrix
from $N_{f_0}$: \emph{QR factorization with optimal column
  pivoting}. This computes matrices $Q_0,R_0$ and a permutation
$p = (p_1, \ldots, p_{\# E_0})$ of the columns of $N_{f_0}$ such that
$N_{f_0}[:,p] = Q_0R_0$, where $Q_0 \in \C^{\gamma \times \gamma}$ is
a unitary matrix, $R_0 \in \C^{\gamma \times \# E_0}$ is upper
triangular and $N_{f_0}[:,p]$ is $N_{f_0}$ with its columns permuted
according to $p$.
The leftmost $\gamma$ columns of $R_0$ form the square, upper
triangular matrix $\hat{R}_0$. The column permutation $p$ is such that
columns $p_1, \ldots, p_\gamma$ form a well-conditioned submatrix of
$N_{f_0}$. In line \ref{line:basischoice}, these columns are selected
to form the matrix $N_{f_0,\B}$. Using the identities
$N_{f_0,\B}^* M_{g}^* = N_{g,\B}^*$
, where ${}^*$ is the conjugate transpose, 
and $N_{f_0,\B} = Q_0 \hat{R}_0$,
we see that the solution $Q_0^*M_{g}^*Q_0$ to the linear system
$\hat{R}_0^* X = N_{g,\B}^*Q_0$ is similar to the matrix $M_g^*$ in
this section, and it can be obtained by back substitution since
$R_0^*$ is lower triangular. Since we extract the coordinates of the
roots form the eigenvalues, not the eigenvectors, we may work with
$Q_0^*M_{g}^*Q_0$ as well. This is exploited in line
\ref{line:linsys}. In line \ref{line:otheralg}, we invoke Algorithm
\ref{alg:geteigenspace}. Lines \ref{line:startif}-\ref{line:endif} are
a straightforward implementation of Remark \ref{rem:computeeval}. As
pointed out, in the case $m = 1$, $\lambda_{ij}$ can alternatively be
computed as a Rayleigh quotient.
  \begin{algorithm}[h!]
    \caption{\textsc{Solve}}
    \label{alg:solve}
    \begin{algorithmic}[1]
    \small
      \Require{An admissible tuple $(\f,A_0,(E_0,\bm{E}),D)$ with $A_0 = \{ \alpha_1 = 0, \alpha_2, \ldots, \alpha_k\}$}
      \Ensure{A \emph{candidate} set of solutions of $\f$, containing all solutions in $(\C \setminus \{0\})^n$}
      \State{ $O \gets $ random matrix of size $( \sum_{i = 1}^s \# E_i ) \times \#D$ }
      \State{ $\textrm{MO} \gets \Mat(\f,\bm{E};D) \cdot O $ } \label{line:compress}
      \State{ Compute $\Cok(\f,\bm{E};D)$ via the SVD of
        $\textrm{MO}$} \label{line:leftnull}
      \State{$f_0 \gets$ random element in $R_{A_0}$ }
      \State{
        $N_{f_0} \leftarrow $ matrix of size $\gamma \times (\# E_0)$ given by $\Cok(\f,\bm{E};D) \cdot
        \Mat(f_0,E_0;D).$}
      \State{$Q_0,R_0,p \gets$ apply QR decomposition with optimal pivoting to $N_{f_0}$} \label{line:qr}
      \State{$\hat{R}_0 \gets$ square, upper triangular matrix given by the first $\gamma$ columns of $R_0$}
      \State{$\B \leftarrow $ exponents in $E_0$ corresponding to columns $p_1, \ldots, p_\gamma$ of $N_{f_0}$} \label{line:basischoice}
     \For{$j = 1, \ldots, k$}
 \State{
        $N_{x^{\alpha_j},\B} \leftarrow \Cok(\f,\bm{E};D) \cdot
        \Mat(x^{\alpha_j},\B;D)$}
      \State{$M_{x^{\alpha_j}}^* \gets $ solve $\hat{R}_0^* X = N_{x^{\alpha_j},\B}^*Q_0$ for $X$ by back substitution} \label{line:linsys}
          \EndFor
          \State $M_g \gets$ random linear combination of $M_{x^\alpha}, \alpha \in A_0$
          \State $\{ (\mu_1, V_{\mu_1}), \ldots, (\mu_\D, V_{\mu_\D})\} \gets$ distinct eigenvalues of $M_g$ and corresponding left eigenspaces
          \State $M_{h} \gets$ a different random linear combination of $M_{x^\alpha}, \alpha \in A_0$
          \State $Z \gets \{\}$
          \For{$i = 1, \ldots, \D$}
          \State ${\cal V} \gets$ \textsc{getEigenspace}($\mu_i,V_{\mu_i},M_{h}$) \label{line:otheralg}
          \If{${\cal V} \neq \{0\}$} \label{line:startif}
          \State $m \gets$ number of rows of ${\cal V}$
          \State $T \gets$ random matrix of size $\gamma \times m$
                    
          \For{$j = 1, \ldots, k$} \label{line:computeLambda}
          \State $\tilde{M} \gets {\cal V} M_{x^{\alpha_j}} T ({\cal V}T)^{-1}$
          \State $\lambda_{ij} \gets \text{trace}(\tilde{M})/m$ \hfill (when $m = 1$, use Rayleigh quotient, see \Cref{rem:computeeval})   \label{line:computeLambdaij}
          \EndFor
          \State $\z_i \gets$ if possible, compute the coordinates of $\z_i$ via \eqref{eq:coordsfromev} and \eqref{eq:evs}
          \State $ Z \gets  Z \cup \{\z_i\}$
          \EndIf \label{line:endif}
          \EndFor
          \State \textbf{return} $Z$
	\end{algorithmic}
      \end{algorithm}

      \begin{remark}
        Alternatively, by \Cref{thm:evthm}, when ${\cal V}$ is
        one-dimensional, we may check if, for a vector $v \in {\cal V}$,
        there is a non-zero constant $C$ and $\z \in \C^n$ such that
        $C \cdot \z^D = v_\z \cdot \Cok(\f,\bm{E};D)$. If $0 \in D$ and
        $\Z D = \Z^n$, we scale $v$ such that $C = 1$ and find $\z$ from
        $\z^D$ as above. When the matrices $M_g$ are multiplication
        operators, this approach is usually referred as the \emph{eigenvector
          criterion} \cite{auzinger_elimination_1988}.
        This idea can be extended to the case where $\cal V$ has dimension $> 1$. Extracting vectors of the form $\z^D$ from a vector space can be viewed as a \emph{harmonic retrieval} problem, see \cite[Sec.~3.3]{vanderstukken2017systems}.
  
      \end{remark}

      \begin{theorem}
        [Correctness]

        Algorithm~\ref{alg:solve} computes a set of points containing every solution of
        $\f$ in the algebraic torus $(\C \setminus \{0\})^n$.
      \end{theorem}
      
      \begin{proof}
        As our input is an admissible tuple, the compatibility
        condition implies that the the matrix $N_{f_0}$ is
        well-defined. By the rank condition and the fact that $f_0$ is generic, $N_{f_0}$ is
        has full rank. See the discussion below
        \Cref{thm:propertiesOFh}. Hence, the matrices $M_g$ and $M_h$
        are well-defined and agree with the ones defined in
        \eqref{eq:defMg}.
        Let $\z_1$ be a solution of $\f$ such that
        $\zeta_1 \in (\C \setminus \{0\})^n$.
        As $\z_1^D \neq 0$, by \Cref{thm:evthm}, we can assume with no
        loss of generality that $\mu_1 = \frac{g}{f_0}(\z_1)$.
        Let
        ${\cal V} := \textsc{getEigenspace}(\mu_1,V_{\mu_1},M_{h})$,
        for generic $h \in R_{A_0}$. As $h$ is generic, by
        \Cref{prop:multeigenval}, all vectors in $\cal V$ belong to the same eigenvalue of $M_{x^{\alpha_j}}$, for $j = 1, \ldots, k$.
        Hence, by \Cref{prop:eigenvalcrit}, there is a non-zero
        constant $C \in \C$ such that the element $\lambda_{1,j}$
        computed in line \ref{line:computeLambdaij} agrees with
        $C \, \zeta_1^{\alpha_j}$, for $\alpha_j \in A_0$.
        Observe that, as $\zeta_1 \in (\C \setminus \{0\})^n$,
        $\lambda_{1,j} \neq 0$.
        Therefore, as the admissible tuple satisfies the lattice
        condition $\Z \, A_0 = \Z^n$, we can recover the
        coordinates of $\zeta_1$ using \eqref{eq:evs} and
        $\zeta_1 \in Z$.
      \end{proof}

      \begin{remark}
        By adapting the previous proof, we observe that
        algorithm~\ref{alg:solve} computes every solution $\zeta$ of
        $\f$ such that $\zeta^D \neq 0$ and the formula given in
        \eqref{eq:evs} is well-defined. That is, for every
        $j \in \{1,\dots,k\}$ and $\ell \in \{1,\dots,n\}$,
        $m_{j,\ell} \geq 0$ or $\zeta^{\alpha_j} \neq 0$.
      \end{remark}

It is clear that the size of the matrices in Algorithm \ref{alg:solve}
depends on the cardinality of the exponent sets in the admissible
tuple. Constructing admissible tuples for certain families of
polynomial systems is an active field of research, strongly related to
the study of \emph{regularity} of ideals in polynomial rings, in the
sense of commutative algebra
\cite[Sec.~20.5]{eisenbud_commutative_2004}. Recent progress in this
area, for the case where $n = s$, was made in
\cite{bender_toric_2020}.
In the next section, we will summarize some of these results by
explicitly describing some admissible tuples for systems with
important types of structures.

As mentioned above, the matrices $M_g$
considered in this section play the role of
\emph{multiplication operators} in the algebra $R/I$, where $I$ is the
ideal generated by the polynomials in $\F$
\cite[Ch. 2]{cox_using_2005}. In the very general setting we consider
here, assuming only that $(\F, A_0, (E_0,{\bm E}), D)$ is an
admissible tuple, the matrices $M_g$ do not necessarily represent such
multiplication operators. However, under some extra assumptions, they
do commute. In this case, we can simplify \Cref{alg:solve} by
computing the simultaneous Schur factorization of
$(M_{x^\alpha})_{\alpha \in A_0}$ as in
\cite[Sec.~3.3]{bender_toric_2020}.
      \begin{theorem}[Criterion for commutativity] \label{thm:critCommute}
    Let $(\f, A_0, (E_0, \bm{E}), D)$ be an admissible tuple and $\gamma = \HF(\f,\bm{E};D)$. Let $f_0$ be such that $N_{f_0}$ satisfies the Rank condition (Definition \ref{def:condSolv}). If
        \[\HF((f_0^2,\f),(E_0,E_1+A_0,\dots,E_s+A_0);D+A_0) -
        \HF(\f,(E_1+A_0,\dots,E_s+A_0);D+A_0) = \gamma,\]
        then for every $g_1,g_2 \in R_{A_0}$ and every valid choice of $B \subset A_0$, we have that 
        $M_{g_1} \, M_{g_2} = M_{g_2} \, M_{g_1}$.
      \end{theorem}

      \begin{proof} 
        In what follows, we fix two vector spaces
        $I_D := \Im(\Sylv_{(\f,(E_1,\dots,E_s);D)})$ and
        $I_{D+A_0} := $ \linebreak
        $\Im(\Sylv_{(\f,(E_1+A_0,\dots,E_s+A_0);D+A_0)})$.  Observe
        that, for every $g \in R_{A_0}$ and $f \in I_D$,
        $g \, f \in I_{D+A_0}$. We write $\B = \{b_1,\dots,b_\gamma\}$
        and given $v \in \C^\gamma$, we set
        $v \cdot \B := \sum_i v_i \, x^{b_i}$.
    
    In this proof, for each $g \in R_{A_0}$, we consider the map
$\tilde{M_g} := N_{f_0,\B}^{-1} \cdot M_g \cdot N_{f_0,\B}$.
    The maps $M_g$ and $\tilde{M_g}$ are similar, so it is enough to
prove that $\tilde{M}_{g_1} \, \tilde{M}_{g_2} = \tilde{M}_{g_2} \,
\tilde{M}_{g_1}$.
It is not hard to show that for $v,w \in \C^\gamma$, such that
$\tilde{M}_g(v) = w \in \C^\gamma$, we have
$g \, (v \cdot \B) \equiv f_0 \, (w \cdot \B)$ modulo $I_D$.

First, observe that, for every $v \in \C^\gamma$,
$g_1 \, g_2 \, (v \cdot \B) \equiv f_0^2 \, ((\tilde{M}_{g_1} \,
\tilde{M}_{g_2} \, v) \cdot B)$ modulo $I_{D+A_0}$. Indeed,
$g_1 \, (v \cdot \B) = f_0 \, ((\tilde{M}_{g_1} \, v) \cdot B) +
h_{1}$, for $h_1 \in I_{D}$ and for $w = \tilde{M}_{g_1} \, v$, we
have that
$g_2 \, (w \cdot \B) = f_0 \, ((\tilde{M}_{g_2} \, w) \cdot B) +
h_{2}$, for $h_2 \in I_{D}$.  Hence,
$g_1 \, g_2 \, (v \cdot \B) = g_2 \, f_0 \, ((\tilde{M}_{g_1} \, v)
\cdot B) + g_2 \, h_1 = f_0^2 \, ((\tilde{M}_{g_2} \, \tilde{M}_{g_1}
\, v) \cdot B) + f_0 \, h_2 + g_2 \, h_1$. As
$f_0 \, h_2 + g_2 \, h_1 \in I_{D + A_0}$, the claim follows.
Since $g_1g_2 = g_2g_1$, it also holds that 
$g_1 \, g_2 \, (v \cdot \B) \equiv f_0^2 \, ((\tilde{M}_{g_2} \,
\tilde{M}_{g_1} \, v) \cdot B)$ modulo $I_{D+A_0}$.

Second, we show that $\{ f_0^2 \, x^{b_i} : b_i \in \B\}$ is a basis of
the vector space $V$ spanned by $\{ f_0^2 \, x^{e} : e \in E_0\}$
modulo $I_{D+A_0}$.
By construction of $\B\subset E_0$, $\{ f_0 \, x^{b_i} : b_i \in \B\}$
is a basis of the vector space spanned by
$\{ f_0 \, x^{e} : e \in E_0\}$ modulo $I_{D}$, so
$\{ f_0^2 \, x^{b_i} : b_i \in \B\}$ generates $V$.
Moreover, by the assumption on the difference of coranks, the dimension of the vector space $V$ is
$\gamma = \# \{f_0^2 \, x^{b_i} : b_i \in \B\}$.

By the first observation, we have that
$f_0^2 \, \left(\left(\left(\tilde{M}_{g_1} \, \tilde{M}_{g_2} -
      \tilde{M}_{g_2} \, \tilde{M}_{g_1}\right) \, v \right) \cdot
  \B\right) \equiv 0$ modulo $I_{D+A_0}$ for every $v \in \C^\gamma$.
By the second observation, the elements in
$\{ f_0^2 \, x^{b_i} : b_i \in \B\}$ are linearly independent modulo
$I_{D+A_0}$. Therefore, we have that
$\tilde{M}_{g_1} \, \tilde{M}_{g_2} = \tilde{M}_{g_2} \,
\tilde{M}_{g_1}$.
  \end{proof}
  
  \begin{remark}
    This criterion is similar to Bayer and Stillman's criterion to
    compute the Castelnuovo-Mumford regularity of ideals defining a
    zero dimensional projective scheme
    \cite[Thm.~1.10]{bayer_criterion_1987}. Under further assumptions
    on $A_0,E_0,D, \text{ and } \B$, the commutativity of the matrices
    implies that $\gamma$ is the number of isolated solutions of the
    system $\F$, see \cite[Thm.~3.1]{mourrain_new_1999}.
  \end{remark}

  \begin{example}
    [Cont.]
    \Cref{thm:critCommute} is independent of the chosen basis $B$. 
    Its hypotheses are not satisfied by the admissible tuple of our running example, as
    \[\HF((f_0^2,\f),(E_0,E_1+A_0,\dots,E_s+A_0);D+A_0) -
    \HF(\f,(E_1+A_0,\dots,E_s+A_0);D+A_0) = 1 < 2.\]
  However, as we showed in \Cref{ex:eigenvectors} for $B = \{1,x\}$,
  the matrices in $\cal M$ do commute.
  \end{example}

  \section{Construction of admissible tuples} \label{sec:constr}

  In this section, we fix an $s$-tuple of sets of exponents
  $\bm{A} := (A_1,\dots,A_s)$, where $A_i \subset \N^n$, and consider
  a polynomial system
  $\f = (f_1,\dots,f_s) \in R_{A_1} \times \dots \times R_{A_s}$.
  We construct tuples that are admissible under mild assumptions on
  $\f$ (\Cref{ass:zerodim}).
  This allows us to compute the solutions of the system $\f$ using
  \Cref{alg:solve}. \Cref{subsec:explicit} states explicit formulas
  for admissible tuples that in practice are near-optimal in the case
  where $s = n$. In the overdetermined case $(s > n)$, we can obtain
  admissible tuples leading to smaller matrices by using incremental
  constructions. These are the topic of Subsection
  \ref{sec:incConstr}.

  The section uses the following notation. The \emph{convex hull} of a
  finite subset $E \subset \R^n$ is the polytope
  $\conv(E) \subset \R^n$ defined as,
  \[
  \conv(E) := \left\{\sum_{e \in E} \lambda_e \, e ~:~ \sum_{e \in E}
    \lambda_e = 1, \lambda_e \geq 0 \text{ for all } e \in E
  \right\}.
  \]
  By a \emph{lattice polytope} we mean a convex polytope
  $P \subset \R^n$ that arises as $\conv(E)$, where $E \subset
  \N^n$. Such a lattice polytope is called \emph{full-dimensional} if
  it has a positive Euclidean volume in $\R^n$.
  Given two polytopes $P_1,P_2 \subset \R^n$ and $c \in \N$, we denote
  by $P_1 + P_2$ the \emph{Minkowski sum} of $P_1, P_2$ and by
  $c \cdot P_1$ the \emph{$c$-dilation} of $P_1$, that is,
\[
P_1 + P_2 := \{\alpha + \beta : \alpha \in P_1, \beta \in P_2\},
\qquad c \cdot P_1 := \{c \, \alpha : \alpha \in P_1\}.
\]
We denote the Cartesian product of two subsets $P_1 \subset \R^{n_1}$
and $P_2 \subset \R^{n_2}$ by
$P_1 \times P_2 := \{(\alpha,\beta) : \alpha \in P_1,\beta \in P_2\}
\subset \R^{n_1} \times \R^{n_2} = \R^{n_1 + n_2}$. Throughout, we use
the notation
$\Delta_n = \conv( \{0, e_1, \ldots, e_n \}) \subset \R^n$ for the
standard simplex in $\R^n$.

\begin{example} \label{ex:minksum} Consider the sets of exponents
  $E_1 = \{(0,0), (1,0), (1,1), (2,0), (0,1)\}$ and \linebreak
  $E_2 = \{(0,0), (1,0), (0,1)\}$. In Figure \ref{fig:minksum},
  the polytopes $P_1 := \conv(E_1)$, $P_2 := \conv(E_2)$, and
  $P_1+P_2 \subset \R^2$ are displayed. Observe that $P_2$
  is the two-dimensional standard simplex $\Delta_2$.
  \end{example}
  
  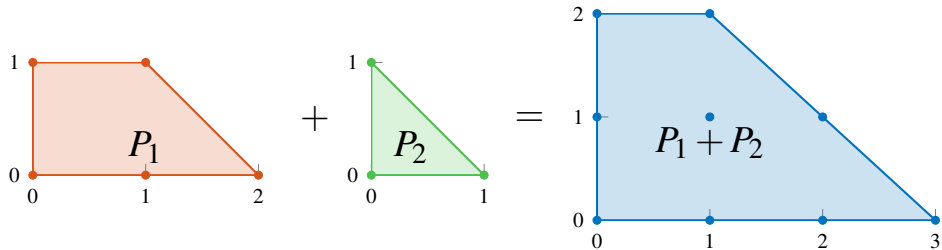
\begin{figure}[h!]
  \centering
  \definecolor{mycolorRed}{rgb}{0.85000,0.32500,0.09800}
  \definecolor{mycolorGreen}{rgb}{0.3,0.75,0.3}
  \definecolor{mycolorBlue}{rgb}{0.00000,0.44700,0.74100}
\begin{tikzpicture} 
  \node at (3.75cm,0.75cm) {$+$};
  \node at (6.60cm,0.75cm) {$=$};
  \begin{axis}[%
    width=3cm,
    height=1.5cm,
    scale only axis,
    xmin=0,
    xmax=2,
    xtick = {0, 1, 2},
    ymin=0,
    ymax=1,
    ytick = {0,1},
    tick label style={font=\scriptsize},
    axis background/.style={fill=white},
    axis line style={draw=none},
    axis x line*=bottom,
    axis y line*=left
    ]
    \addplot [color=mycolorRed,solid,thick, fill opacity = 0.2, fill = mycolorRed,forget plot]
  table[row sep=crcr]{%
0	0\\
2	0\\
1	1\\
0	1\\
0	0\\
};
\addplot[only marks,mark=*,mark size=1.5pt,mycolorRed
        ]  coordinates {
    (0,0) (1,0) (1,1) (2,0) (0,1)
};

\node at (axis cs:1,0.25) {$P_1$};
\end{axis}
\begin{axis}[%
    xshift=4.5cm,
    width=1.5cm,
    height=1.5cm,
    scale only axis,
    xmin=0,
    xmax=1,
    xtick = {0, 1},
    ymin=0,
    ymax=1,
    ytick = {0,1},
    tick label style={font=\scriptsize},
    axis line style={draw=none} ,
    axis x line*=bottom,
    axis y line*=left
]

\addplot [color=mycolorGreen,solid,thick, fill opacity = 0.2, fill = mycolorGreen,forget plot]
  table[row sep=crcr]{%
0	0\\
1	0\\
0	1\\
0	0\\
};
\addplot[only marks,mark=*,mark size=1.5pt,mycolorGreen
        ]  coordinates {
    (0,0) (1,0) (0,1)
};

\node at (axis cs:0.35,0.25) {$P_2$};
\end{axis}
\begin{axis}[%
    yshift=-.6cm,
    xshift=7.5cm,
    width=4.5cm,
    height=2.75cm,
    scale only axis,
    xmin=0,
    xmax=3,
    xtick = {0, 1, 2, 3},
    ymin=0,
    ymax=2,
    ytick = {0,1, 2},
    tick label style={font=\scriptsize},
    axis line style={draw=none} ,
    axis x line*=bottom,
    axis y line*=left
    ]
    \addplot [color=mycolorBlue,solid,thick, fill opacity = 0.2, fill = mycolorBlue,forget plot]
  table[row sep=crcr]{%
0	0\\
3	0\\
1	2\\
0	2\\
0	0\\
};
\addplot[only marks,mark=*,mark size=1.5pt,mycolorBlue
        ]  coordinates {
    (0,0) (1,0) (2, 0) (3, 0) (0,1) (1,1) (2, 1) (0,2) (1,2) (0,3) (1,3)
};

\node at (axis cs:1,.75) {$P_1 + P_2$};
\end{axis}
\end{tikzpicture}%
\caption{Polytopes from Example \ref{ex:minksum}.}
\label{fig:minksum}
\end{figure}

\subsection{Explicit constructions}  \label{subsec:explicit}
We present explicit constructions of admissible tuples for the following types of polynomial systems, listed in (more or less) increasing order of generality. 
  \begin{enumerate}
  \item \textbf{Dense systems.} These are systems for which $f_i$ may involve all monomials of degree at most $d_i$, where $(d_1, \ldots, d_s) \in \N_{>0}^s$ is an $s$-tuple of positive natural numbers. For dense systems, we have $A_i = \{ \alpha \in \N^n ~:~ \alpha_1 + \cdots + \alpha_n  \leq d_i \} = (d_i \cdot \Delta_n) \cap \N^n$. 
  \item \textbf{Unmixed systems.}
  We say that the polynomial system $\f$ is \emph{unmixed} if there is
  a full-dimensional lattice polytope $P$ and integers $d_1,\dots,d_s$ such that
  $d_i \cdot P = \conv(A_i)$.
  The \emph{codegree} of $P$ is the smallest $t \in \N_{>0}$ such that %
  $ t \cdot P$ contains a point with integer coordinates in its interior.
  Note that dense systems can be viewed as unmixed systems with $P = \Delta_n$.
  
  \item \textbf{Multi-graded dense systems.} A different, natural generalization of the dense case allows different degrees for different subgroups of the variables $x_1, \ldots, x_n$. Let $\{ \I_1, \ldots, \I_r \}$ be a partition of $\{1, \ldots, n \}$, i.e. $\I_j \subset \{1, \ldots, n \}$, $\I_j \cap \I_k = \emptyset$ and $\bigcup_{j=1}^r \I_j = \{1, \ldots, n\}$. This way we obtain subsets $x_{\I_1}, \ldots, x_{\I_r} \subset \{x_1, \ldots, x_n\}$ of the variables, indexed by the $\I_j$. In a multi-graded dense system, $f_i$ may contain all monomials of degree at most $d_{i,j}$ in the variables $x_{\I_j}$. If the variables are ordered such that the first $n_1$ variables are indexed by $\I_1$, the next $n_2$ variables by $\I_2$ and so on, this means $A_i = ((d_{i,1} \cdot \Delta_{n_1}) \times \cdots \times (d_{i,r} \cdot \Delta_{n_r})) \cap \N^n$. Necessarily we have $n_1 + \cdots + n_r = n$. A dense system is a multi-graded dense system with~$r = 1$.
  \item \textbf{Multi-unmixed systems.}
  This is a generalization of the unmixed and the multi-graded dense case, where there are full-dimensional lattice polytopes
  $P_1 \subset \R^{n_1},\dots,P_r \subset \R^{n_r}$ such that
  $0 \in P_i$ and $n = \sum_i n_i$ and for each $i \in \{1, \ldots, s\}$, an $r$-tuple $(d_{i,1},\dots,d_{i,r}) \in \N^r$ such that
  $\conv(A_i) = (d_{i,1} \cdot  P_1) \times \dots \times
  (d_{i,r} \cdot  P_r)$. That is, the convex hull of $A_i$ is the product of dilations of the polytopes $P_1, \ldots, P_r$.
  Note that a multi-graded dense system is a multi-unmixed system with $P_i = \Delta_{n_i}$, and an unmixed system is a multi-unmixed system with~$r~=~1$. 

\item \textbf{Mixed systems.} This is the most general case, our only assumption on each $A_i$ is that the lattice polytope $\sum_{i=1}^s \conv(A_i) \subset \R^n$ is full-dimensional.
\end{enumerate}

If the full-dimensionality requirements in the previous list are not
fulfilled, one can reformulate the system using fewer variables. For
polynomial systems from these nested families, admissible tuples are
presented in Table \ref{tbl:paramsTori}. In what follows, we discuss
them in more detail.

The tuples presented in \Cref{tbl:paramsTori} are admissible under a
\emph{zero-dimensionality assumption} on the system
$\f$. Unfortunately, it is not enough to require that $\f(x) = 0$ has
finitely many solutions in $\C^n$ or $(\C \setminus \{0\})^n$. Loosely
speaking, we need that the lifting of $\f$ to a certain larger
solution space has finitely many solutions. This is best understood in
the context of toric geometry. We refer the reader to
\cite[Sec.~3]{telen2019numerical} or \cite[Sec.~2]{bender_toric_2020}
for a description of the zero-dimensionality assumption in this
language. Here, we omit terminology from toric geometry and state the
assumption in terms of \emph{face systems}, following
\cite{bernshtein_number_1975}. We will use the notation 
\begin{equation} \label{eq:notationf}
 \f = (f_1, \ldots, f_s) \in R_{A_1} \times \dots \times R_{A_s} , \quad f_i = \sum_{\alpha \in A_i} c_{i,\alpha} \, x^\alpha.
\end{equation} 
For any vector $v \in \R^n$, we define 
\[ A_{i,v} := \{ \alpha \in A_i ~:~ \pair{v,\alpha} = \min_{\beta \in A_i} \pair{v,\beta} \},\]
where $\pair{v,\alpha} = v_1 \alpha_1 + \cdots + v_n \alpha_n \in \R$. For $i = 1, \ldots, s$, fix any $\beta_{i,v} \in A_{i,v}$. This gives a new system 
\[\F_v = (f_{1,v}, \ldots, f_{s,v}) \in R_{A_{1,v}} \times \dots \times R_{A_{s,v}} , \quad f_{i,v} = \sum_{\alpha \in A_{i,v}} c_{i,\alpha} \,  x^{\alpha-\beta_{i,v}}, \]
called the \emph{face system} associated to $v$. 
The exponents $\alpha - \beta_{i, v}, \alpha \in A_i$ occurring in the polynomials $f_{i,v}$ lie in a lattice of rank $< n$ when $v \neq 0$. We denote this lattice by
\[ M_v = \Z \cdot \left \{ \bigcup_{i=1}^s (A_{i,v} - \beta_{i,v}) \right \}. \]
Let $r_v$ be the rank of $M_v$. Applying a change of coordinates, $\F_v$ is a system of Laurent polynomials in $r_v$ variables on the torus $(\C \setminus \{0\})^{r_v}$. Its solutions are independent of the choice of $\beta_{i,v} \in A_{i,v}$.

\begin{assumption}[Zero-dimensionality assumption] \label{ass:zerodim}
For every $v \in \R^n$, the \emph{face system} $\f_v(x) = 0$ has finitely many (possibly zero) solutions in $(\C \setminus \{0\})^{r_v}$. 
\end{assumption}
Setting $v = 0$, Assumption \ref{ass:zerodim} implies that $\f(x) = 0$ has finitely many solutions in $(\C \setminus \{0\})^n$.
{
\begin{table}[h!]
\small
  \[
\def\arraystretch{1.35}
  \begin{array}{c | c }
    \hline
    \multicolumn{2}{c}{
    \phantom{\text{ (\cite[Cor.~4.3]{bender_toric_2020})}} \hfill
    \text{1. Dense case, } 
    \conv(A_i) = d_{i} \cdot \Delta_n
    \hfill \text{ (\cite[Cor.~4.3]{bender_toric_2020})}
    } 
    \\  \hline
    A_0 &  \phantom{((d_0 = 1)} \hfill \Delta_n \cap \N^n  \hfill (d_0 = 1) \\
    E_i & ((\sum_{j \neq i} d_{j} - n) \cdot \Delta_n) \cap \N^n \\
    D &  ((\sum_{j } d_{j} - n) \cdot \Delta_n) \cap \N^n \\ \hline
    \hline
    \multicolumn{2}{c}{
    \phantom{\text{ (\cite[Thm.~4.5]{bender_toric_2020})}} \hfill
    \text{2. Unmixed case, } 
    \conv(A_i) = d_{i} \cdot P
    \hfill \text{ (\cite[Thm.~4.5]{bender_toric_2020})}} 
    \\  \hline
    A_0 &  \phantom{((d_0 = 1)} \hfill P \cap \N^n  \hfill (d_0 = 1) \\
    E_i & ((\sum_{j \neq i} d_{j} - \textsc{Codegree}(P) + 1) \cdot P) \cap \N^n \\
    D &  ((\sum_{j } d_{j} - \textsc{Codegree}(P) + 1) \cdot P) \cap \N^n \\ \hline
    \hline
    \multicolumn{2}{c}{
    \qquad\qquad
    \text{3. Multi-graded dense case, } 
    \conv(A_i) = (d_{i,1} \cdot \Delta_{n_1} \times \dots \times d_{i,r} \cdot \Delta_{n_r})
    \qquad
    \text{ (\cite[Ex.~10]{bender_toric_2020})}} 
    \\  \hline
    A_0 &
          \phantom{(d_{0,k} = 1)} \hfill
          (\Delta_{n_1} \times \dots \times \Delta_{n_r}) \cap \N^n
            \hfill (d_{0,k} = 1)

    \\
    E_i & \prod_k ((\sum_{j \neq i} d_{j,k} - n_k) \cdot \Delta_{n_k}) \cap \N^n \\
    D &        \prod_k ((\sum_{j} d_{j,k} - n_k) \cdot \Delta_{n_k}) \cap \N^n \\ \hline
    \hline
    \multicolumn{2}{c}{
    \phantom{\text{ (\cite[Ex.~10]{bender_toric_2020})}} \hfill 
    \text{4. Multi-unmixed case, } 
        \conv(A_i) = (d_{i,1} \, P_1 \times \dots \times d_{i,r} \, P_r) 
    \hfill \text{ (\cite[Ex.~10]{bender_toric_2020})}} 
    \\  \hline
    A_0 &
          \phantom{(d_{0,k} = 1)} \hfill
          (P_1 \times \dots \times P_r) \cap \N^n
            \hfill (d_{0,k} = 1)

    \\
    E_i & \prod_k ((\sum_{j \neq i} d_{j,k} - \textsc{Codegree}(P_k) + 1) \cdot P_k) \cap \N^n \\
    D &        \prod_k ((\sum_{j } d_{j,k} - \textsc{Codegree}(P_k) + 1) \cdot P_k) \cap \N^n \\ \hline
    \hline
    \multicolumn{2}{c}{
    \phantom{\text{ (\cite[Thm.~4.4]{bender_toric_2020})}} \hfill
    \text{5. Mixed case, } 
    \conv(A_i) = P_i
    \hfill \text{ (\cite[Thm.~4.4]{bender_toric_2020})}
    } 
    \\  \hline
    A_0 &  \phantom{(P_0 = \Delta_n)} \hfill \Delta_n \cap \N^n  \hfill (P_0 = \Delta_n) \\
    E_i & (\sum_{j \neq i} P_j) \cap \N^n \\
    D &  (\sum_{j } P_j) \cap \N^n \\ \hline
  \end{array}
  \]
  \caption{Admissible tuples for five families of structured
    polynomial systems. In the table, we assume that all $d_i
    > 0$, $d_{i,j} \geq 0$ and $P\subset
    \R^n, P_i \subset
    \R^{n_i}$ are full dimensional lattice polytopes.}
  \label{tbl:paramsTori}
\end{table}
}
  
    \begin{remark}
    \label{rmk:notMuchSolutionsAtInf}
    \Cref{ass:zerodim} holds for a generic element $\f \in R_{A_1}
    \times \cdots \times
    R_{A_s}$, in the sense of \Cref{sec:alg}. In fact, for a generic system
    $\f$ all face systems $\f_v$ for $v \neq
    0$ have no solutions in $(\C \setminus
    \{0\})^{r_v}$, and the condition for this to hold only depends on the
    coefficients associated to some vertices of the polytopes
    $\conv(A_i)$, see \cite{canny_optimal_1991}. The fact that we can
    allow finitely many solutions for all face systems comes from the
    recent contributions
    \cite{telen2019numerical,bender_toric_2020}. In practice, this
    means that our algorithm is robust in the presence of isolated
    solutions \emph{at} or \emph{near infinity} (where this is
    understood in the appropriate toric sense).
  \end{remark}
  
  \begin{theorem}
    \label{thm:overdeterminedSparse}
    Consider a polynomial system $\F = (f_1, \ldots, f_s)$ with supports $A_1, \ldots, A_s$ satisfying Assumption \ref{ass:zerodim}. Consider
    $(A_0,(E_0,\dots,E_s),D)$ as defined in
    Table~\ref{tbl:paramsTori}.
    Then, we have that $(\F,A_0,(E_0,\dots,E_s),D)$ is an admissible
    tuple.
  \end{theorem}

  \begin{proof}
    We sketch the proof. We need to show that the three conditions in
    \Cref{def:condSolv} are satisfied. Observe that, by construction,
    the elements from the tuple satisfy the Compatibility condition
    and $A_0$ satisfies the Lattice condition.
    By \Cref{ass:zerodim}, for generic $f_0 \in R_{A_0}$, the system
    $(f_0,\dots,f_s)$ has no solutions on the toric variety associated
    to the lattice polytope $\conv(A_1) + \cdots + \conv(A_s)$ and we
    can adapt \cite[Thm.~4.3]{bender_toric_2020} straightforwardly to
    the case of no solutions (the Koszul complex of sheaves in that
    proof is exact by
    \cite[Ch.~2.B,~Prop.1.4.a]{gelfand_discriminants_1994}, see also
    \cite[Thm.~3.C]{massri2016solving}). Hence, following the same
    procedure as in \cite[Sec.~4]{bender_toric_2020}, we can show that
    $\HF((f_0,\f),(E_0,\bm{E});D) = 0$.
    Therefore, by \Cref{thm:propertiesOFh}, the Rank condition holds.
  \end{proof}

  \begin{remark}[The number $\gamma$ and the number of solutions]\label{rmk:dimCoKer}
    Consider $\f$ satisfying \Cref{ass:zerodim}. We fix an admissible
    tuple $(\F,A_0,(E_0,\dots,E_s),D)$ from \Cref{tbl:paramsTori}. If
    $n = s$, the dimension $\gamma = \HF(\F, {\bm E}; D)$ is the
    number of solutions defined by $\f$ on the compact toric variety
    $X \supset (\C \setminus \{0\})^n$ from
    \cite[Thm.~4.4]{bender_toric_2020}, counted with
    multiplicities. For generic $\f$, all solutions have multiplicity
    1 and lie in $(\C \setminus \{0\})^n$, which means that $\gamma$
    is the \emph{mixed volume} of the polytopes \linebreak
    $\conv(A_1), \ldots, \conv(A_n)$
    \cite[Thm.~A]{bernshtein_number_1975}. Additionally, in these
    cases, we have a complete characterization of the invariant
    subspaces of $M_g$ as it represents a multiplication operator, see
    \cite[Sec.~3.2]{bender_toric_2020}.
    It was pointed out to us by Laurent Bus\'e that
    \cite[Lem.~6.2]{chardin_powers_2013} should imply that the same
    holds for $s > n$, see the proof of
    \cite[Prop.3]{buse_multigraded_2021} for an example of how to
    prove such a result in the multihomogeneous case.
  \end{remark}
  
  Macaulay matrices defined by the tuples from \Cref{tbl:paramsTori}
  have been used in different algorithms for solving sparse polynomial
  systems, e.g.~sparse resultants \cite{emiris1999matrices}, truncated
  normal forms \cite{telen2018stabilized}, \Groebner bases
  \cite{bender_towards_2018,bender2019gr}, and others
  \cite{massri2016solving}. When restricted to Macaulay matrices,
  these constructions are often near-optimal when $s = n$.
  However, there exist other kind of smaller matrices which can be
  also used to solve the system
  \cite{bender_determinantal_2021,bender_bilinear_2018}.
  When $s > n$, we can often work with much smaller Macaulay
  matrices. This is the topic of the next subsection.

  \subsection{Incremental constructions}
  \label{sec:incConstr}
  Even though the tuples from \Cref{thm:overdeterminedSparse} are
  admissible, they might lead to the construction of unnecessarily big
  matrices in \Cref{alg:solve}. To avoid this, we present an
  incremental approach which leads to the construction of potentially
  smaller matrices. For ease of exposition, we consider only the
  unmixed case. The ideas can be extended to the other cases.
  
  In what follows, we fix a polytope $P$ such that $0 \in P$ and
  integers $d_0,\dots,d_s \in \N_{>0}$. We consider sets of exponents
  $A_0,A_1,\dots,A_s \subset \N^n$ such that, for each
  $i \in \{0,\dots,s\}$, we have
  $\conv(A_i) = d_i \cdot P$.
  For each $\lambda \in \N$, we define ${\bm E}^\lambda = (E_1^\lambda, \ldots, E_s^\lambda)$ with
\begin{equation} \label{eq:expsunmixed}
\quad E_i^{\lambda} :=  
                            \left\{
                            \begin{array}{c l}
                              ((\lambda - d_i) \cdot P) \cap \N^n & \text{if } \lambda \geq d_i \\
                              \emptyset & \text{otherwise} 
                            \end{array}                            
                                     \right. \quad \text{ for } i = 0, \ldots, s, \quad \text{and } \quad D^{\lambda} := (\lambda \cdot P) \cap \N^n.
\end{equation}
  \begin{theorem}
    \label{thm:boundUnmixed}
    With the above notation, consider an unmixed polynomial system
    $\F \in R_{A_1} \times \cdots \times R_{A_s}$, with
    $\conv(A_i) = d_i \cdot P$, satisfying \Cref{ass:zerodim}. For any
    $\lambda \in \N$ such that there is $f_0 \in R_{A_0}$ satisfying
    $\rank(N_{f_0}) = \HF(\F,{\bm E}^\lambda;D^\lambda)$, we have that
    the tuple $(\F,A_0,(E_0^{\lambda},{\bm E}^\lambda),D^{\lambda})$
    is admissible.
    Moreover, for any
    $\lambda \geq \sum_i d_i - \textsc{Codegree}(P) + 1$ and for
    generic $f_0 \in R_{A_0}$, we have that
    $\rank(N_{f_0}) = \HF(\F,{\bm E}^\lambda;D^\lambda)$.
  \end{theorem}

  \begin{proof}
    By construction, the tuple satisfies the Compatibility and Lattice
    conditions. By assumption, it satisfies the Rank condition, so it
    is admissible.
    The proof follows as in \Cref{thm:overdeterminedSparse}.
  \end{proof}

  The bound upper bound on $\lambda$ obtained in
  \Cref{thm:boundUnmixed} is not tight for overdetermined systems. Below, we will present a broad class of overdetermined
  unmixed systems, namely semi-regular* sequences, for which we can
  improve it.

  \begin{remark}[The number $\gamma$ and the number of solutions]
    \label{rmk:nsolsSpecialBound}
    In contrast to \Cref{rmk:dimCoKer}, the condition
    $\rank(N_{f_0}) = \Cok(\F,{\bm E}^\lambda;D^\lambda)$ does not
    imply that $\gamma = \HF(\F,{\bm E}^\lambda;D^\lambda)$ agrees
    with the number of solutions of $\F$ on some toric
    compactification.
    In fact, in \Cref{sec:numexp}, we will present examples of
    semi-regular* sequences (\Cref{deg:semireg}) where $\gamma$ is
    strictly larger than the number of solutions.
    In these cases, the matrices $M_g$ from \Cref{eq:defMg} are not
    multiplication operators.
    For readers familiar with the concept of Castelnuovo-Mumford
    regularity, we note that this happens because the degree
    $D^\lambda$ belongs to the regularity of $\{f_0,\F\}$, but not
    necessarily to that of $\F$.
  \end{remark}

  \Cref{thm:boundUnmixed} suggests an algorithm for finding an
  admissible tuple for an unmixed system $\F$: we simply check, for a
  random element $f_0 \in R_{A_0}$ and increasing values of $\lambda$,
  whether $\rank(N_{f_0}) = \HF(\F,{\bm E}^\lambda;D^\lambda)$ with
  $N_{f_0} = \Cok(\f,{\bm E}^\lambda; D^\lambda) \cdot
  \Mat(\f,E_0^\lambda;D^\lambda)$. In order to do this efficiently,
  instead of computing $\Cok(\f,{\bm E}^{\lambda+1}; D^{\lambda+1})$
  directly as the left nullspace of the large matrix
  $\Mat(\f,{\bm E}^{\lambda+1}; D^{\lambda+1})$, we will obtain it
  from the previously computed
  $\Cok(\f,{\bm E}^{\lambda}; D^{\lambda})$ and a smaller Macaulay
  matrix. This technique was applied in the dense setting
  ($P = \Delta_n$) in
  \cite{batselier2014fast,mourrain2019truncated}, where it is also called
  `degree-by-degree' approach. See also \cite{parkinson2021analysis}
  for a recent complexity analysis.

Note that, by construction, $E_i^\lambda \subset E_i^{\lambda+1}$ and $D^\lambda \subset D^{\lambda+1}$. The first step is to construct the following $2 \times 2$ block matrix
\begin{equation} \label{eq:prodmtx}
 (\Cok(\f,{\bm E}^{\lambda}; D^{\lambda}) \times \id) := \begin{blockarray}{cc}
D^\lambda & D^{\lambda+1} \setminus D^\lambda \\
\begin{block}{[cc]}
\Cok(\f,{\bm E}^{\lambda}; D^{\lambda}) & 0 \\ 
0 & \id \\
\end{block}
\end{blockarray}.
\end{equation}
Here $\id$ denotes the identity matrix of size $\#(D^{\lambda+1} \setminus D^\lambda )$. Note that the columns of the matrix $(\Cok(\f,{\bm E}^{\lambda}; D^{\lambda}) \times \id)$ are indexed by $D^{\lambda+1}$, where the first block column is indexed by $D^\lambda \subset D^{\lambda+1}$. Next, we set ${\bm E}^{\lambda+1} \setminus {\bm E}^\lambda := (E_1^{\lambda + 1} \setminus E_1^\lambda, \ldots, E_s^{\lambda + 1} \setminus E_s^\lambda)$ and construct the Macaulay matrix $\Mat(\f, {\bm E}^{\lambda+1} \setminus {\bm E}^\lambda; D^{\lambda+1})$. Here we require that the ordering of the rows is compatible with the ordering of the columns in \eqref{eq:prodmtx}. Let $L^{\lambda + 1}$ be a left nullspace matrix of the matrix product
\begin{equation} \label{eq:smallermtx}
(\Cok(\f,{\bm E}^{\lambda}; D^{\lambda}) \times \id)  \cdot \Mat(\f, {\bm E}^{\lambda+1} \setminus {\bm E}^\lambda; D^{\lambda+1}).
\end{equation}
Then
$\Cok(\f,{\bm E}^{\lambda+1}; D^{\lambda+1}) = L^{\lambda + 1} \cdot
(\Cok(\f,{\bm E}^{\lambda}; D^{\lambda}) \times \id)$ is a left
nullspace matrix for the Macaulay matrix
$\Mat(\f,{\bm E}^{\lambda+1}; D^{\lambda+1})$.
The power of this
approach lies in the fact that \eqref{eq:smallermtx} is much smaller
than $\Mat(\f,{\bm E}^{\lambda+1}; D^{\lambda+1})$, which leads to a
much cheaper left nullspace computation.

This gives an iterative algorithm for \emph{updating} the left
nullspace matrix $\Cok(\f,{\bm E}^{\lambda}; D^{\lambda})$. We start
our iteration by considering $\lambda = \max_i d_i$, as we want to
take into account all of the equations. This discussion is summarized
in Algorithm \ref{alg:getATunmixed}. Note that the algorithm computes
the cokernel $\Cok(\f,{\bm E}; D)$ for the admissible tuple as a
by-product, as well as the matrix $N_{f_0}$. This allows us to skip the
steps before line \ref{line:qr} in \Cref{alg:solve}.

\begin{algorithm}[h]
    \caption{\textsc{GetAdmissibleTupleUnmixed}}
    \label{alg:getATunmixed}
    \begin{algorithmic}[1]
    \small
    	\Require{An unmixed system $\F$ satisfying Assumption \ref{ass:zerodim}, the polytope $P \ni 0$, the degrees $(d_1, \ldots, d_s)$.}
    	 \Ensure{An admissible tuple $(\f, A_0, (E_0,{\bm E}), D)$, a left nullspace matrix $\Cok(\f, (E_0,{\bm E}); D)$ and a corresponding matrix $N_{f_0}$}
    	 \State $d_0 \gets 1$
    	 \State $A_0 \gets P \cap \N^n$
    	 \State $f_0 \gets$ a random element of $R_{A_0}$
     \State $\lambda \gets \max_i d_i$
     \State $\Cok(\f,{\bm E}^{\lambda}; D^{\lambda}) \gets $ left nullspace of $\Mat(\f, {\bm E}^\lambda; D^\lambda)$, for the sets of exponents in \eqref{eq:expsunmixed}
     \State $r \gets$ number of rows of $\Cok(\f,{\bm E}^{\lambda}; D^{\lambda})$
     \State $N_{f_0} = \Cok(\f,{\bm E}^{\lambda}; D^{\lambda}) \cdot M(f_0, E_0^\lambda; D^\lambda)$
     \While{$\rank(N_{f_0}) \neq r$}
     \State $(\Cok(\f,{\bm E}^{\lambda}; D^{\lambda}) \times \id) \gets $ the matrix from \eqref{eq:prodmtx}
     \State $L^{\lambda + 1} \gets$ a left nullspace matrix for the matrix in \eqref{eq:smallermtx}
     \State $\Cok(\f,{\bm E}^{\lambda+1}; D^{\lambda+1}) \gets L^{\lambda + 1} \cdot (\Cok(\f,{\bm E}^{\lambda}; D^{\lambda}) \times \id)$
     \State $\lambda \gets \lambda +1$
      \State $r \gets$ number of rows of $\Cok(\f,{\bm E}^{\lambda}; D^{\lambda})$
      \State $N_{f_0} = \Cok(\f,{\bm E}^{\lambda}; D^{\lambda}) \cdot M(f_0, E_0^\lambda; D^\lambda)$
     \EndWhile
     \State {\bf return} $(\f, A_0, (E_0^\lambda,{\bm E}^\lambda), D^\lambda)$, $\Cok(\f,{\bm E}^{\lambda}; D^{\lambda})$, $N_{f_0}$
	\end{algorithmic}
      \end{algorithm}

      \begin{remark}[Other incremental constructions]
        There are alternative incremental constructions for the
        matrices $N_{f_0}$ which also reuse information from previous
        steps to speed up the computations. An example is the F5
        criterion in the context of \Groebner bases
        \cite{faugere2014sparse}. These ideas extend naturally to the mixed
        setting, see \cite{bender2019gr}. However, these
        approaches based on monomial orderings lead to bad numerical
        behaviour. 
        In the context of sparse resultants for mixed systems, Canny
        and Emiris \cite{emiris_efficient_1995} proposed an
        alternative incremental algorithm to construct admissible
        tuples leading to smaller Macaulay matrices. Their procedure
        can be enhanced with the approach followed in this section.
      \end{remark}

 In the rest of this subsection, we identify a broad class of overdetermined unmixed
 systems for which we can obtain smaller admissible tuples than the ones in
 \Cref{thm:overdeterminedSparse}. We will need some more notation. The \emph{Ehrhart series} of a polytope $P$ is the series
  \[
  \textup{ES}_P(t) = \sum_{\lambda \geq 0} \#((\lambda \cdot P) \cap
  \Z^n) \; t^\lambda.
  \]
  The \emph{Hilbert series} of a polynomial system
  $\F_0 := (f_0,\dots,f_s) \in R_{A_0} \times \dots\ \times R_{A_s}$,
 is
  \[
  \textup{HS}_{\F_0}(t) := \sum_{\lambda \geq 0}
  \HF(\F_0,\bm{E}^\lambda;D^\lambda) \; t^\lambda.
  \]

  \begin{definition}[Semi-regularity*] \label{deg:semireg}
    We say that $\F_0$ is a semi-regular* sequence if
    \[
      \textup{HS}_{\F_0}(t) =
  \left[ \textup{ES}_P(t) \, \prod_{i = 0}^s (1 - t^{d_i})
  \right]_+,
    \]
  where $[ \; \cdot \; ]_+$ means that we truncate the series in its
  first negative coefficient.
  \end{definition}

  Observe that we write semi-regular* sequence with an asterisk as the
  usual definition of \emph{semi-regular sequence} asks for this condition
  on the Hilbert series to hold for every subsystem $(f_0,\dots,f_i)$,
  $i \leq s$. However, semi-regular sequences are too restrictive for
  our purposes.
  
  Even in the case where $P$ is a standard
  simplex, semi-regular* sequences are not understood as well as regular
  sequences. For example, Fr\"oberg's conjecture states that being a
  semi-regular* sequence is a generic condition
  \cite{froberg_inequality_1985}. This conjecture, supported by a lot
  of empirical evidence, was extended to the unmixed case
  \cite{faugere2014sparse}.
  
  \begin{theorem}\label{thm:semireg}
    Consider an unmixed polynomial system
    $\F \in R_{A_1} \times \cdots \times R_{A_s}$ and a polynomial
    $f_0 \in R_{A_0}$, with $\conv(A_i) = d_i \cdot P$.
    Let $\lambda_{min}$ be the smallest integer among the degrees of the monomials in
    $\textup{ES}_P(t) \, \prod_{i = 0}^s (1 - t^{d_i})$ standing with a non-positive coefficient.
    We have that, if $(f_0,\F)$ is a semi-regular* sequence, then the
    tuple
    $((f_0,\F),A_0,(E_0^{\lambda_{min}},\bm{E}^{\lambda_{min}});D^{\lambda_{min}})$
    is admissible.
      \end{theorem}

  \begin{proof}
    The proof follows from the fact that
    $\HF((f_0,\F),E_0^{\lambda_{min}},\bm{E}^{\lambda_{min}});D^{\lambda_{min}})
    = 0$ as the sequence is semi-regular*.
  \end{proof}

  It follows directly from \Cref{thm:boundUnmixed} that, whenever
  $(f_0,\F)$ is semi-regular*,
  $\lambda_{min} \leq \sum_i d_i - \textsc{Codegree}(P) + 1$.
  In \Cref{sec:ex:over}, we present generic families of
  zero-dimensional overdetermined systems $\F$ such that $(f_0,\F)$ is
  semi-regular*. For these systems, we show that the previous
  inequality can be strict.
  
  Semi-regular* sequences give us an inexpensive heuristic to discover
  values for $\lambda$ for which we can obtain admissible tuples. It
  was observed in practice
  \cite{bardet2005asymptotic,faugere2014sparse} that for many systems
  $\F$ not having much solutions outside the torus (see
  \Cref{rmk:notMuchSolutionsAtInf}), they can be extended to
  semi-regular* sequences. Moreover, there are asymptotic estimates
  for the expected value of $\lambda$ \cite{bardet2005asymptotic}.

\section{Experiments} \label{sec:numexp}

In this section we illustrate several aspects of the methods presented
in this paper via numerical experiments. We implemented these
algorithms in the new Julia package \texttt{EigenvalueSolver.jl},
which is freely available at \link. For all computations
involving polytopes, we use \texttt{Polymake.jl} (version 0.5.3),
which is a Julia interface to Polymake \cite{kaluba2020polymake}. We
compare our results with the package \texttt{HomotopyContinuation.jl}
(version 2.3.1), which is state-of-the-art software for solving
systems of polynomial equations using homotopy continuation
\cite{breiding2018homotopycontinuation}.  All computations were run on
a 16 GB MacBook Pro with an Intel Core i7 processor working at 2.6
GHz.

To evaluate the quality of a numerical approximation $\z \in \C^n$ of
a solution for a polynomial system $\f$ given by \eqref{eq:notationf}.
We define the \emph{backward error} $\BWE(\z)$ of $\z$ as
\begin{equation} \label{eq:BWE}
 \BWE(\z) = \frac{1}{s} \sum_{i=1}^s \frac{|f_i(\z)|}{\sum_{\alpha \in A_i} | c_{i,\alpha} \z^\alpha | + 1} . 
 \end{equation}
This error can be interpreted as a measure for the relative distance
of $\F$ to a system $\F'$ for which $\F'(\z) = 0$, see
\cite[App.~C]{telen2020thesis}.

Additionally, we validate our computed solutions via
\emph{certification}. For that, we use the certification procedure
implemented in the function \texttt{certify} of
\texttt{HomotopyContinuation.jl}, which is based on interval
arithmetic, as described in \cite{breiding2020certifying}.
This function takes as an input a list of approximate solutions to
$\f$ and tries to compute a list of small boxes in $\C^n$, each of
them containing an approximate input solution and exactly one \emph{actual}
solution to $\f$. The total number of connected components in the union of these boxes is denoted by $\crt$ in what follows. Each of these connected components contains exactly one solution of $\f$, and one or more approximate input solutions. This means that $\crt$ is a lower bound on the number of solutions to $\f$. If $\crt$ equals the number of solutions, the solutions of $\f$ are in one-to-one correspondence with the approximate input solutions. In this case, we say that all solutions are \emph{certified}. The function \texttt{certify}
assumes that $\f$ is \emph{square}, i.e.~$\f$ should have as many
equations as variables ($s = n$). If this is not the case ($s > n$),
we use \texttt{certify} on a system obtained by taking $n$ random
$\C$-linear combinations of $f_1, \ldots, f_s$.

The main function of our package \texttt{EigenvalueSolver.jl} is
\texttt{solve\_EV}, which implements Algorithm \ref{alg:solve}. It
takes as an input an admissible tuple (see \Cref{def:condSolv}).
This tuple can be computed using the auxiliary functions provided in
our implementation, which are tailored to take into account the
specific structure of the systems. These functions use the explicit
and incremental constructions from Section \ref{sec:constr}.

It is common in applications that we have to solve many different
generic systems $\f$ with the same supports $A_1, \ldots, A_s$. In
this case, the computation of the admissible tuple can be seen as an
\emph{offline} computation that needs to happen only once. We will
therefore report both the \emph{offline} and the \emph{online} computation time. The offline computation time is the time needed for computing an admissible tuple \emph{and} executing \texttt{solve\_EV}. The \emph{online}
computation re-uses a previously computed admissible tuple to execute \texttt{solve\_EV}. 

\Cref{tab:notation} summarizes the notation that we use to describe
our experiments.

\begin{table}
\centering
\small
\begin{tabular}{l|l}
$n$ & number of variables \\
$\delta$ & number of solutions \\
$\crt$ & number of connected components computed by \texttt{certify} \\
$t_{\textup{on}}$ & online computation time in seconds \\
$t_{\textup{off}}$ & offline computation time in seconds  \\
$\BWE$ & maximum backward error of all computed approximate solutions \\
$\overline{\BWE}$ & geometric mean of the backward errors of all computed solutions \\
$\gamma$ & the number of rows of $\Cok(\f, {\bm E}; D)$ \\
$\# D$ & cardinality of $D$, i.e.~the number of columns of $\Cok(\f, {\bm E}; D)$
\end{tabular}
\caption{Notation in the experiments in Section \ref{sec:numexp}.}
\label{tab:notation}
\end{table}

The section is organized as follows. In \Cref{sec:ex:sqSys}, we
consider square systems $(s = n)$ and show how to use
\texttt{EigenvalueSolver.jl} to solve them. 
In \Cref{sec:ex:over}, we solve overdetermined systems
$(s > n)$ using our incremental algorithm. We
perform several experiments summarized in \Cref{tab:oddense} and \Cref{tab:odunmixed}.  In
\Cref{subsec:solatinf}, we consider systems for which one solutions drifts off to `infinity'. In \Cref{sec:ex:hom}, we
compare our algorithm with homotopy continuation methods.

\subsection{Square systems}
\label{sec:ex:sqSys}

In this subsection, we demonstrate some of the functionalities of
\texttt{EigenvalueSolver.jl} by solving square systems, that is
$s = n$, for each of the families in Table \ref{tbl:paramsTori}.
The code used for the examples can be found at \link \, in the Jupyter
notebook
\href{https://github.com/simontelen/JuliaEigenvalueSolver/blob/main/examples/demo_EigenvalueSolver.ipynb}{\texttt{/example/demo\_EigenvalueSolver.ipynb}}.
We fix the parameters of Table \ref{tbl:paramsTori} and consider
specific supports $A_i$ as described below. We construct random
polynomial systems by assigning random real coefficients to each of
the monomials, which we draw from a standard normal distribution. By
Remark \ref{rmk:dimCoKer}, the number $\gamma$ equals the number of
solutions $\delta$ for all examples in this subsection.

For our first example, we intersect two degree 20 curves in the
plane. That is, we consider a square, dense system $\f_1$ with $n = 2$
and $d_1 = d_2 = 20$. The equations are generated by the following
simple commands:
\begin{verbatim}
@polyvar x[1:2]; ds = [20;20];
f = EigenvalueSolver.getRandomSystem_dense(x, ds)
\end{verbatim}
By B\'ezout's theorem, this system has $\delta = 400$ different
solutions, which we can compute via
\begin{verbatim}
sol = EigenvalueSolver.solve_CI_dense(f, x; DBD = false)
\end{verbatim}
In the previous line, the option \verb|DBD = false| indicates that we
do not want to use the `degree-by-degree' approach for solving this
system, that is, the incremental approach described in \Cref{sec:incConstr}.
Experiments show that this strategy is only beneficial for square systems with $n \geq
3$. The letters \texttt{CI} in the name of the function stand for
\emph{complete intersection}, which indicates that a zero-dimensional
square system is expected as its input.
In this example, we have $\#D = 820$ and the computation took
$t_{\textup{off}} = 0.83$ seconds.  To validate the solutions, we
compute their backward errors.
\begin{verbatim}
BWEs = EigenvalueSolver.get_residual(f, sol, x)
\end{verbatim}
The maximal value, computed using the command
\texttt{maximum(BWEs)}, is $\BWE \approx 10^{-12}$. The function
\texttt{certify} from \texttt{HomotopyContinuation.jl} certifies
$\crt = 400$ distinct solutions.
If we perform the same computation with parameters $n = 3$,
$(d_1,d_2,d_3) = (4,8,12)$, we obtain
$ \delta = \gamma = \crt = 384,~ \#D = 2300,~ t_{\textup{off}} =
3.10,~ \BWE \approx 10^{-11}$.
\begin{verbatim}
@polyvar x[1:3]; ds = [4;8;12];
f = EigenvalueSolver.getRandomSystem_dense(x, ds)
sol = EigenvalueSolver.solve_CI_dense(f, x)
\end{verbatim}

For our next example, we consider an unmixed system $\f_2$ with
parameters
\begin{equation}
\label{eq:parametersunmixed}
n = 2, \quad P = \conv(A), \quad A =
\{(0,0),(1,0),(1,1),(0,1),(2,2) \}, \quad (d_1,d_2) = (5,12).
\end{equation} 
The following code executes this example,
\begin{verbatim}
@polyvar x[1:2]; A = [0 0; 1 0; 1 1; 0 1; 2 2]; d = [5;12]; 
f = EigenvalueSolver.getRandomSystem_unmixed(x, A, d)
sol, A0, E, D = EigenvalueSolver.solve_CI_unmixed(f, x, A, d)
\end{verbatim}
In this case, we obtain
$ \delta = \gamma = \crt = 240,~ \#D = 685,~ t_{\textup{off}} = 0.94,~
\BWE \approx 10^{-11}$.
We remark that the function \texttt{solve\_CI\_unmixed} also returns
the admissible tuple $(A_0, {\bm E}, D)$, so that it can be used to
solve another generic unmixed system with the same
parameters, without redoing the polyhedral computations to generate
this tuple. This can be done in the following way,
\begin{verbatim}
sol = EigenvalueSolver.solve_EV(f, x, A0, E, D; check_criterion = false)
\end{verbatim}
The option \texttt{check\_criterion = false} in the previous line
indicates that the input tuple is admissible, so we do not need to
spend time on checking whether the criterion in
\Cref{thm:propertiesOFh} is satisfied.
Using this option, the \emph{online} computation is faster and takes $t_{\textup{on}} = 0.41$ seconds, yet the parameters $\delta, \gamma, \crt, \BWE$ are comparable to the offline case. 
To illustrate how the unmixed function exploits the structure of the
equations, in Figure \ref{fig:densevsunmixed}, we plot the exponents
in $D$ for this example, together with the exponents in $D$ for our
dense example $\f_1$. In both plots, we have highlighted the exponents in the set $B$ that were selected using QR factorization with optimal column pivoting. These monomial bases clearly do not correspond to any standard (Gr\"obner) or border basis. Figure \ref{fig:densevsunmixed} should be compared to, for instance, Figure 2 in \cite{telen2018stabilized}.
\begin{figure}[h!]
\centering
\footnotesize
\definecolor{mycolorGreen}{rgb}{0.3,0.75,0.3}
\begin{tikzpicture} 

    \begin{axis}[%
    width=5cm,
    height=5cm,
    scale only axis,
    xmin= 0,
    xmax= 40,
    ymin= 0,
    ymax= 40,
    axis background/.style={fill=white},
    ]
    \addplot[only marks,mark=*,mark size=1.5pt,mycolor1!30!white
            ]  coordinates { 
(39,0) (38,1) (37,2) (36,3) (35,4) (34,5) (33,6) (32,7) (31,8) (30,9) (29,10) (28,11) (27,12) (26,13) (25,14) (24,15) (23,16) (22,17) (21,18) (20,19) (19,20) (18,21) (17,22) (16,23) (15,24) (14,25) (13,26) (12,27) (11,28) (10,29) (9,30) (8,31) (7,32) (6,33) (5,34) (4,35) (3,36) (2,37) (1,38) (0,39) (38,0) (37,1) (36,2) (35,3) (34,4) (33,5) (32,6) (31,7) (30,8) (29,9) (28,10) (27,11) (26,12) (25,13) (24,14) (23,15) (22,16) (21,17) (20,18) (19,19) (18,20) (17,21) (16,22) (15,23) (14,24) (13,25) (12,26) (11,27) (10,28) (9,29) (8,30) (7,31) (6,32) (5,33) (4,34) (3,35) (2,36) (1,37) (0,38) (37,0) (36,1) (35,2) (34,3) (33,4) (32,5) (31,6) (30,7) (29,8) (28,9) (27,10) (26,11) (25,12) (24,13) (23,14) (22,15) (21,16) (20,17) (19,18) (18,19) (17,20) (16,21) (15,22) (14,23) (13,24) (12,25) (11,26) (10,27) (9,28) (8,29) (7,30) (6,31) (5,32) (4,33) (3,34) (2,35) (1,36) (0,37) (36,0) (35,1) (34,2) (33,3) (32,4) (31,5) (30,6) (29,7) (28,8) (27,9) (26,10) (25,11) (24,12) (23,13) (22,14) (21,15) (20,16) (19,17) (18,18) (17,19) (16,20) (15,21) (14,22) (13,23) (12,24) (11,25) (10,26) (9,27) (8,28) (7,29) (6,30) (5,31) (4,32) (3,33) (2,34) (1,35) (0,36) (35,0) (34,1) (33,2) (32,3) (31,4) (30,5) (29,6) (28,7) (27,8) (26,9) (25,10) (24,11) (23,12) (22,13) (21,14) (20,15) (19,16) (18,17) (17,18) (16,19) (15,20) (14,21) (13,22) (12,23) (11,24) (10,25) (9,26) (8,27) (7,28) (6,29) (5,30) (4,31) (3,32) (2,33) (1,34) (0,35) (34,0) (33,1) (32,2) (31,3) (30,4) (29,5) (28,6) (27,7) (26,8) (25,9) (24,10) (23,11) (22,12) (21,13) (20,14) (19,15) (18,16) (17,17) (16,18) (15,19) (14,20) (13,21) (12,22) (11,23) (10,24) (9,25) (8,26) (7,27) (6,28) (5,29) (4,30) (3,31) (2,32) (1,33) (0,34) (33,0) (32,1) (31,2) (30,3) (29,4) (28,5) (27,6) (26,7) (25,8) (24,9) (23,10) (22,11) (21,12) (20,13) (19,14) (18,15) (17,16) (16,17) (15,18) (14,19) (13,20) (12,21) (11,22) (10,23) (9,24) (8,25) (7,26) (6,27) (5,28) (4,29) (3,30) (2,31) (1,32) (0,33) (32,0) (31,1) (30,2) (29,3) (28,4) (27,5) (26,6) (25,7) (24,8) (23,9) (22,10) (21,11) (20,12) (19,13) (18,14) (17,15) (16,16) (15,17) (14,18) (13,19) (12,20) (11,21) (10,22) (9,23) (8,24) (7,25) (6,26) (5,27) (4,28) (3,29) (2,30) (1,31) (0,32) (31,0) (30,1) (29,2) (28,3) (27,4) (26,5) (25,6) (24,7) (23,8) (22,9) (21,10) (20,11) (19,12) (18,13) (17,14) (16,15) (15,16) (14,17) (13,18) (12,19) (11,20) (10,21) (9,22) (8,23) (7,24) (6,25) (5,26) (4,27) (3,28) (2,29) (1,30) (0,31) (30,0) (29,1) (28,2) (27,3) (26,4) (25,5) (24,6) (23,7) (22,8) (21,9) (20,10) (19,11) (18,12) (17,13) (16,14) (15,15) (14,16) (13,17) (12,18) (11,19) (10,20) (9,21) (8,22) (7,23) (6,24) (5,25) (4,26) (3,27) (2,28) (1,29) (0,30) (29,0) (28,1) (27,2) (26,3) (25,4) (24,5) (23,6) (22,7) (21,8) (20,9) (19,10) (18,11) (17,12) (16,13) (15,14) (14,15) (13,16) (12,17) (11,18) (10,19) (9,20) (8,21) (7,22) (6,23) (5,24) (4,25) (3,26) (2,27) (1,28) (0,29) (28,0) (27,1) (26,2) (25,3) (24,4) (23,5) (22,6) (21,7) (20,8) (19,9) (18,10) (17,11) (16,12) (15,13) (14,14) (13,15) (12,16) (11,17) (10,18) (9,19) (8,20) (7,21) (6,22) (5,23) (4,24) (3,25) (2,26) (1,27) (0,28) (27,0) (26,1) (25,2) (24,3) (23,4) (22,5) (21,6) (20,7) (19,8) (18,9) (17,10) (16,11) (15,12) (14,13) (13,14) (12,15) (11,16) (10,17) (9,18) (8,19) (7,20) (6,21) (5,22) (4,23) (3,24) (2,25) (1,26) (0,27) (26,0) (25,1) (24,2) (23,3) (22,4) (21,5) (20,6) (19,7) (18,8) (17,9) (16,10) (15,11) (14,12) (13,13) (12,14) (11,15) (10,16) (9,17) (8,18) (7,19) (6,20) (5,21) (4,22) (3,23) (2,24) (1,25) (0,26) (25,0) (24,1) (23,2) (22,3) (21,4) (20,5) (19,6) (18,7) (17,8) (16,9) (15,10) (14,11) (13,12) (12,13) (11,14) (10,15) (9,16) (8,17) (7,18) (6,19) (5,20) (4,21) (3,22) (2,23) (1,24) (0,25) (24,0) (23,1) (22,2) (21,3) (20,4) (19,5) (18,6) (17,7) (16,8) (15,9) (14,10) (13,11) (12,12) (11,13) (10,14) (9,15) (8,16) (7,17) (6,18) (5,19) (4,20) (3,21) (2,22) (1,23) (0,24) (23,0) (22,1) (21,2) (20,3) (19,4) (18,5) (17,6) (16,7) (15,8) (14,9) (13,10) (12,11) (11,12) (10,13) (9,14) (8,15) (7,16) (6,17) (5,18) (4,19) (3,20) (2,21) (1,22) (0,23) (22,0) (21,1) (20,2) (19,3) (18,4) (17,5) (16,6) (15,7) (14,8) (13,9) (12,10) (11,11) (10,12) (9,13) (8,14) (7,15) (6,16) (5,17) (4,18) (3,19) (2,20) (1,21) (0,22) (21,0) (20,1) (19,2) (18,3) (17,4) (16,5) (15,6) (14,7) (13,8) (12,9) (11,10) (10,11) (9,12) (8,13) (7,14) (6,15) (5,16) (4,17) (3,18) (2,19) (1,20) (0,21) (20,0) (19,1) (18,2) (17,3) (16,4) (15,5) (14,6) (13,7) (12,8) (11,9) (10,10) (9,11) (8,12) (7,13) (6,14) (5,15) (4,16) (3,17) (2,18) (1,19) (0,20) (19,0) (18,1) (17,2) (16,3) (15,4) (14,5) (13,6) (12,7) (11,8) (10,9) (9,10) (8,11) (7,12) (6,13) (5,14) (4,15) (3,16) (2,17) (1,18) (0,19) (18,0) (17,1) (16,2) (15,3) (14,4) (13,5) (12,6) (11,7) (10,8) (9,9) (8,10) (7,11) (6,12) (5,13) (4,14) (3,15) (2,16) (1,17) (0,18) (17,0) (16,1) (15,2) (14,3) (13,4) (12,5) (11,6) (10,7) (9,8) (8,9) (7,10) (6,11) (5,12) (4,13) (3,14) (2,15) (1,16) (0,17) (16,0) (15,1) (14,2) (13,3) (12,4) (11,5) (10,6) (9,7) (8,8) (7,9) (6,10) (5,11) (4,12) (3,13) (2,14) (1,15) (0,16) (15,0) (14,1) (13,2) (12,3) (11,4) (10,5) (9,6) (8,7) (7,8) (6,9) (5,10) (4,11) (3,12) (2,13) (1,14) (0,15) (14,0) (13,1) (12,2) (11,3) (10,4) (9,5) (8,6) (7,7) (6,8) (5,9) (4,10) (3,11) (2,12) (1,13) (0,14) (13,0) (12,1) (11,2) (10,3) (9,4) (8,5) (7,6) (6,7) (5,8) (4,9) (3,10) (2,11) (1,12) (0,13) (12,0) (11,1) (10,2) (9,3) (8,4) (7,5) (6,6) (5,7) (4,8) (3,9) (2,10) (1,11) (0,12) (11,0) (10,1) (9,2) (8,3) (7,4) (6,5) (5,6) (4,7) (3,8) (2,9) (1,10) (0,11) (10,0) (9,1) (8,2) (7,3) (6,4) (5,5) (4,6) (3,7) (2,8) (1,9) (0,10) (9,0) (8,1) (7,2) (6,3) (5,4) (4,5) (3,6) (2,7) (1,8) (0,9) (8,0) (7,1) (6,2) (5,3) (4,4) (3,5) (2,6) (1,7) (0,8) (7,0) (6,1) (5,2) (4,3) (3,4) (2,5) (1,6) (0,7) (6,0) (5,1) (4,2) (3,3) (2,4) (1,5) (0,6) (5,0) (4,1) (3,2) (2,3) (1,4) (0,5) (4,0) (3,1) (2,2) (1,3) (0,4) (3,0) (2,1) (1,2) (0,3) (2,0) (1,1) (0,2) (1,0) (0,1) (0,0)  };

\addplot[only marks,mark=*,mark size=1.5pt,mycolor1!70!black
            ]  coordinates { 
(0,38) (0,0) (0,1) (2,36) (2,0) (0,2) (4,34) (3,0) (0,37) (0,36) (4,0) (0,3) (6,32) (0,35) (5,0) (38,0) (0,4) (8,30) (0,34) (0,5) (10,28) (0,6) (6,0) (2,35) (0,7) (2,2) (0,33) (0,8) (7,0) (0,32) (2,34) (36,2) (0,9) (13,25) (2,3) (4,33) (37,0) (2,33) (0,31) (8,0) (0,10) (36,0) (34,4) (15,23) (0,30) (1,0) (0,11) (35,0) (0,29) (32,6) (0,13) (6,31) (9,0) (2,4) (0,28) (30,8) (17,21) (0,26) (4,32) (0,12) (35,2) (1,2) (27,11) (2,5) (4,2) (2,32) (0,14) (0,27) (0,15) (34,0) (8,29) (0,16) (2,6) (4,31) (33,0) (33,4) (34,2) (25,13) (10,27) (2,7) (20,18) (0,23) (0,25) (4,3) (22,16) (2,8) (32,0) (6,30) (0,17) (10,0) (11,27) (2,31) (4,1) (12,25) (0,24) (31,0) (11,0) (33,2) (14,0) (2,9) (12,0) (13,0) (1,37) (14,23) (0,22) (32,4) (15,0) (4,4) (31,6) (9,27) (7,2) (1,1) (1,36) (29,9) (6,2) (0,18) (16,21) (5,33) (30,0) (0,20) (2,11) (2,30) (5,2) (1,3) (6,29) (1,11) (11,25) (37,1) (4,5) (32,2) (1,35) (4,30) (8,2) (31,4) (18,20) (29,0) (0,19) (31,2) (29,8) (7,29) (1,34) (1,4) (26,11) (4,6) (24,14) (0,21) (9,29) (30,2) (4,29) (28,0) (1,5) (16,0) (24,13) (4,7) (2,1) (2,29) (33,5) (3,9) (1,7) (3,35) (30,6) (1,13) (14,24) (26,0) (1,33) (36,1) (17,0) (5,32) (1,8) (3,1) (8,27) (2,28) (20,17) (1,14) (5,1) (1,6) (18,19) (7,31) (9,2) (28,7) (28,10) (2,27) (22,15) (5,4) (27,0) (35,3) (13,23) (1,15) (1,32) (1,9) (12,26) (18,0) (25,11) (6,4) (26,9) (3,33) (35,1) (3,34) (34,3) (1,26) (16,22) (28,9) (6,5) (7,30) (4,28) (26,12) (1,17) (7,1) (29,4) (27,9) (1,25) (3,2) (1,10) (1,27) (11,26) (24,0) (34,1) (6,1) (31,7) (16,20) (1,31) (1,23) (30,4) (5,31) (8,1) (3,11) (1,20) (1,12) (22,0) (1,18) (3,32) (21,17) (29,2) (23,13) (1,30) (6,28) (19,0) (24,11) (1,29) (8,5) (5,3) (29,6) (8,3) (7,5) (1,21) (9,28) (5,5) (9,1) (33,3) (3,28) (33,1) (30,7) (21,15) (23,15) (10,2) (1,16) (1,22) (3,31) (32,1) (3,3) (28,4) (19,19) (10,25) (1,19) (3,8) (11,2) (25,12) (22,13) (1,28) (3,12) (32,5) (31,1) (4,8) (3,30) (12,23) (3,4) (27,10) (32,3) (5,30) (5,8) (25,2) (7,26) (28,6) (13,24) (23,14) (3,6) (23,0) (3,10) (8,26) (9,5) (1,24) (6,6) (25,7) (31,5) (3,29) (10,26) (10,1) (29,7) (8,25) (6,8) (14,1) (6,7) (15,22) (7,3) (31,3) (7,28) (29,1) (27,8) (11,1) (26,2) (6,3) (8,28) (2,10) (5,28) (9,26) (30,3) (10,5) (29,3) (10,3) (30,1) (2,15) (5,29) (27,2) (5,6) (12,24) (2,16) (9,25) (21,16) (29,5) (25,10) (3,7) (30,5) (3,27) (28,8) (5,26) (12,2) (3,5) (25,1) (15,1) (15,21) (2,23) (28,1) (9,6) (2,22) (2,25) (26,10) (9,23) (24,12) (8,6) (2,12) (23,12) (17,20) (28,2) (11,24) (7,4) (21,0) (7,27) (2,26) (28,3) (13,3) (27,5) (28,5) (25,8) (21,14) (10,4) (14,2) (20,16) (27,3) (8,4) (2,21) (7,7) (4,27) (8,21) (5,27) (5,7) (18,1) (27,4) (13,1) (12,1) (24,1) (15,20) (12,4) (26,1) (4,22) (11,5) (5,9) (4,9) (3,23) (27,7) (2,24) (21,1) (25,0) (20,2) (3,21) (9,18) (23,1) (26,3) (15,3) (4,23) (11,7)  };
        
         \end{axis}
         \end{tikzpicture}%
\qquad
\begin{tikzpicture} 

    \begin{axis}[%
    width=5cm,
    height=5cm,
    scale only axis,
    xmin= 0,
    xmax= 40,
    ymin= 0,
    ymax= 40,
    axis background/.style={fill=white},
    ]
    \addplot[only marks,mark=*,mark size=1.5pt,mycolor1!30!white
            ]  coordinates { 
(0,0) (0,1) (0,2) (0,3) (0,4) (0,5) (0,6) (0,7) (0,8) (0,9) (0,10) (0,11) (0,12) (0,13) (0,14) (0,15) (0,16) (0,17) (0,18) (1,0) (1,1) (1,2) (1,3) (1,4) (1,5) (1,6) (1,7) (1,8) (1,9) (1,10) (1,11) (1,12) (1,13) (1,14) (1,15) (1,16) (1,17) (1,18) (2,0) (2,1) (2,2) (2,3) (2,4) (2,5) (2,6) (2,7) (2,8) (2,9) (2,10) (2,11) (2,12) (2,13) (2,14) (2,15) (2,16) (2,17) (2,18) (2,19) (3,0) (3,1) (3,2) (3,3) (3,4) (3,5) (3,6) (3,7) (3,8) (3,9) (3,10) (3,11) (3,12) (3,13) (3,14) (3,15) (3,16) (3,17) (3,18) (3,19) (4,0) (4,1) (4,2) (4,3) (4,4) (4,5) (4,6) (4,7) (4,8) (4,9) (4,10) (4,11) (4,12) (4,13) (4,14) (4,15) (4,16) (4,17) (4,18) (4,19) (4,20) (5,0) (5,1) (5,2) (5,3) (5,4) (5,5) (5,6) (5,7) (5,8) (5,9) (5,10) (5,11) (5,12) (5,13) (5,14) (5,15) (5,16) (5,17) (5,18) (5,19) (5,20) (6,0) (6,1) (6,2) (6,3) (6,4) (6,5) (6,6) (6,7) (6,8) (6,9) (6,10) (6,11) (6,12) (6,13) (6,14) (6,15) (6,16) (6,17) (6,18) (6,19) (6,20) (6,21) (7,0) (7,1) (7,2) (7,3) (7,4) (7,5) (7,6) (7,7) (7,8) (7,9) (7,10) (7,11) (7,12) (7,13) (7,14) (7,15) (7,16) (7,17) (7,18) (7,19) (7,20) (7,21) (8,0) (8,1) (8,2) (8,3) (8,4) (8,5) (8,6) (8,7) (8,8) (8,9) (8,10) (8,11) (8,12) (8,13) (8,14) (8,15) (8,16) (8,17) (8,18) (8,19) (8,20) (8,21) (8,22) (9,0) (9,1) (9,2) (9,3) (9,4) (9,5) (9,6) (9,7) (9,8) (9,9) (9,10) (9,11) (9,12) (9,13) (9,14) (9,15) (9,16) (9,17) (9,18) (9,19) (9,20) (9,21) (9,22) (10,0) (10,1) (10,2) (10,3) (10,4) (10,5) (10,6) (10,7) (10,8) (10,9) (10,10) (10,11) (10,12) (10,13) (10,14) (10,15) (10,16) (10,17) (10,18) (10,19) (10,20) (10,21) (10,22) (10,23) (11,0) (11,1) (11,2) (11,3) (11,4) (11,5) (11,6) (11,7) (11,8) (11,9) (11,10) (11,11) (11,12) (11,13) (11,14) (11,15) (11,16) (11,17) (11,18) (11,19) (11,20) (11,21) (11,22) (11,23) (12,0) (12,1) (12,2) (12,3) (12,4) (12,5) (12,6) (12,7) (12,8) (12,9) (12,10) (12,11) (12,12) (12,13) (12,14) (12,15) (12,16) (12,17) (12,18) (12,19) (12,20) (12,21) (12,22) (12,23) (12,24) (13,0) (13,1) (13,2) (13,3) (13,4) (13,5) (13,6) (13,7) (13,8) (13,9) (13,10) (13,11) (13,12) (13,13) (13,14) (13,15) (13,16) (13,17) (13,18) (13,19) (13,20) (13,21) (13,22) (13,23) (13,24) (14,0) (14,1) (14,2) (14,3) (14,4) (14,5) (14,6) (14,7) (14,8) (14,9) (14,10) (14,11) (14,12) (14,13) (14,14) (14,15) (14,16) (14,17) (14,18) (14,19) (14,20) (14,21) (14,22) (14,23) (14,24) (14,25) (15,0) (15,1) (15,2) (15,3) (15,4) (15,5) (15,6) (15,7) (15,8) (15,9) (15,10) (15,11) (15,12) (15,13) (15,14) (15,15) (15,16) (15,17) (15,18) (15,19) (15,20) (15,21) (15,22) (15,23) (15,24) (15,25) (16,0) (16,1) (16,2) (16,3) (16,4) (16,5) (16,6) (16,7) (16,8) (16,9) (16,10) (16,11) (16,12) (16,13) (16,14) (16,15) (16,16) (16,17) (16,18) (16,19) (16,20) (16,21) (16,22) (16,23) (16,24) (16,25) (16,26) (17,0) (17,1) (17,2) (17,3) (17,4) (17,5) (17,6) (17,7) (17,8) (17,9) (17,10) (17,11) (17,12) (17,13) (17,14) (17,15) (17,16) (17,17) (17,18) (17,19) (17,20) (17,21) (17,22) (17,23) (17,24) (17,25) (17,26) (18,0) (18,1) (18,2) (18,3) (18,4) (18,5) (18,6) (18,7) (18,8) (18,9) (18,10) (18,11) (18,12) (18,13) (18,14) (18,15) (18,16) (18,17) (18,18) (18,19) (18,20) (18,21) (18,22) (18,23) (18,24) (18,25) (18,26) (18,27) (19,2) (19,3) (19,4) (19,5) (19,6) (19,7) (19,8) (19,9) (19,10) (19,11) (19,12) (19,13) (19,14) (19,15) (19,16) (19,17) (19,18) (19,19) (19,20) (19,21) (19,22) (19,23) (19,24) (19,25) (19,26) (19,27) (20,4) (20,5) (20,6) (20,7) (20,8) (20,9) (20,10) (20,11) (20,12) (20,13) (20,14) (20,15) (20,16) (20,17) (20,18) (20,19) (20,20) (20,21) (20,22) (20,23) (20,24) (20,25) (20,26) (20,27) (20,28) (21,6) (21,7) (21,8) (21,9) (21,10) (21,11) (21,12) (21,13) (21,14) (21,15) (21,16) (21,17) (21,18) (21,19) (21,20) (21,21) (21,22) (21,23) (21,24) (21,25) (21,26) (21,27) (21,28) (22,8) (22,9) (22,10) (22,11) (22,12) (22,13) (22,14) (22,15) (22,16) (22,17) (22,18) (22,19) (22,20) (22,21) (22,22) (22,23) (22,24) (22,25) (22,26) (22,27) (22,28) (22,29) (23,10) (23,11) (23,12) (23,13) (23,14) (23,15) (23,16) (23,17) (23,18) (23,19) (23,20) (23,21) (23,22) (23,23) (23,24) (23,25) (23,26) (23,27) (23,28) (23,29) (24,12) (24,13) (24,14) (24,15) (24,16) (24,17) (24,18) (24,19) (24,20) (24,21) (24,22) (24,23) (24,24) (24,25) (24,26) (24,27) (24,28) (24,29) (24,30) (25,14) (25,15) (25,16) (25,17) (25,18) (25,19) (25,20) (25,21) (25,22) (25,23) (25,24) (25,25) (25,26) (25,27) (25,28) (25,29) (25,30) (26,16) (26,17) (26,18) (26,19) (26,20) (26,21) (26,22) (26,23) (26,24) (26,25) (26,26) (26,27) (26,28) (26,29) (26,30) (26,31) (27,18) (27,19) (27,20) (27,21) (27,22) (27,23) (27,24) (27,25) (27,26) (27,27) (27,28) (27,29) (27,30) (27,31) (28,20) (28,21) (28,22) (28,23) (28,24) (28,25) (28,26) (28,27) (28,28) (28,29) (28,30) (28,31) (28,32) (29,22) (29,23) (29,24) (29,25) (29,26) (29,27) (29,28) (29,29) (29,30) (29,31) (29,32) (30,24) (30,25) (30,26) (30,27) (30,28) (30,29) (30,30) (30,31) (30,32) (30,33) (31,26) (31,27) (31,28) (31,29) (31,30) (31,31) (31,32) (31,33) (32,28) (32,29) (32,30) (32,31) (32,32) (32,33) (32,34) (33,30) (33,31) (33,32) (33,33) (33,34) (34,32) (34,33) (34,34) (34,35) (35,34) (35,35) (36,36)  };

\addplot[only marks,mark=*,mark size=1.5pt,mycolor1!70!black
            ]  coordinates { 
(0,17) (17,0) (1,17) (34,34) (17,1) (4,19) (0,0) (1,0) (15,0) (5,19) (32,33) (2,0) (19,4) (8,21) (14,0) (0,1) (28,31) (3,0) (32,31) (5,0) (12,23) (4,0) (6,0) (32,32) (13,0) (30,32) (7,0) (17,2) (32,30) (11,0) (1,16) (19,5) (10,22) (8,0) (21,8) (10,0) (14,24) (16,0) (0,16) (23,12) (12,0) (26,30) (18,2) (28,22) (16,25) (22,10) (16,2) (0,2) (20,27) (15,2) (2,18) (27,20) (33,32) (31,28) (9,0) (24,14) (31,32) (1,15) (6,20) (18,26) (29,24) (30,26) (26,18) (20,6) (28,30) (19,6) (14,2) (33,33) (21,27) (4,18) (24,15) (1,1) (30,30) (22,11) (0,15) (0,3) (18,4) (0,4) (24,29) (27,30) (2,15) (31,30) (15,1) (29,30) (21,9) (11,22) (27,21) (3,18) (30,28) (18,3) (16,4) (0,14) (30,27) (0,7) (25,16) (24,16) (15,24) (12,22) (17,4) (22,12) (3,1) (16,1) (5,18) (26,19) (0,6) (12,1) (31,31) (14,3) (6,1) (27,22) (25,29) (3,2) (31,29) (18,25) (0,11) (22,28) (12,2) (20,10) (14,1) (22,27) (21,10) (4,1) (27,28) (2,17) (0,8) (23,14) (10,1) (8,20) (16,3) (29,28) (0,12) (15,5) (7,1) (2,1) (20,7) (28,23) (2,16) (0,5) (5,17) (8,1) (9,21) (18,5) (23,13) (30,31) (1,2) (10,2) (16,24) (0,10) (29,26) (29,29) (26,28) (30,29) (1,3) (17,3) (5,1) (29,25) (29,31) (4,2) (17,5) (13,1) (4,16) (0,9) (26,26) (25,22) (28,24) (3,17) (7,20) (24,17) (17,25) (25,19) (13,23) (8,2) (5,2) (21,11) (2,2) (1,4) (25,18) (4,17) (15,3) (20,8) (4,3) (0,13) (27,29) (13,2) (10,20) (25,17) (29,27) (11,1) (28,29) (11,21) (19,26) (10,3) (20,11) (7,2) (1,9) (20,12) (14,22) (15,4) (5,3) (15,23) (28,28) (6,19) (9,3) (6,2) (27,27) (2,14) (11,2) (26,23) (18,6) (19,24) (7,19) (2,11) (25,21) (26,27) (19,7) (21,12) (16,6) (26,29) (1,10) (17,24) (26,25) (28,27) (8,3) (25,20) (19,12) (15,22) (25,28) (9,20) (24,20) (4,4) (20,13) (1,14) (1,12) (19,25) (27,26) (23,15) (21,26) (2,8) (27,24) (13,22)  };
         \end{axis}
         \end{tikzpicture}%
\caption{Exponents in $D$, constructed as in Table \ref{tbl:paramsTori}, for the dense systems $\f_1$ (left) and the unmixed system $\f_2$ (right). The exponents are shown as lattice points. Dark coloured dots correspond to the set $\B \subset D$ chosen by the QR factorization in line \ref{line:qr} in Algorithm \ref{alg:solve}.
}
\label{fig:densevsunmixed}
\end{figure}
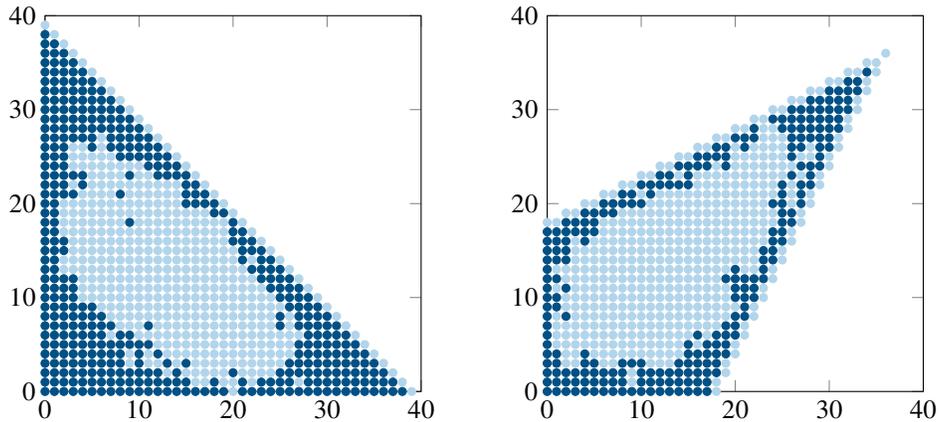

We can solve multi-graded dense and multi-unmixed systems using the implemented
functions \texttt{solve\_CI\_multi\_dense} and
\texttt{solve\_CI\_multi\_unmixed}, respectively.  Table
\ref{tab:mgmu} summarizes our choice of parameters and the results of
our experiments for these systems.
\begin{table}[h!]
\begin{footnotesize}
\begin{tabular}{c|c|c}
family & parameters & results \\ \hline
multi-graded, dense & $ n_1 = n_2 = 2, (d_{j,k}) = \begin{pmatrix}
1 & 6 \\ 2 & 1 \\ 3 &  2 \\ 4  & 1
\end{pmatrix}, $ & $\begin{matrix} \delta = \gamma = \crt = 219,~ \#D = 3025,\\ t_{\textup{off}} = 15.46,~ \BWE \approx 10^{-11} \end{matrix}$ \\ \hline
multi-unmixed & $ n_1 = n_2 = 2, (d_{j,k}) = \begin{pmatrix}
1 & 1 \\ 1 & 1 \\ 1 &  1 \\ 1 & 1
\end{pmatrix}, \begin{matrix} P_1 = \conv(A) \\ P_2 = 2 \cdot \Delta_2 \end{matrix} $ & $\begin{matrix} \delta = \gamma = \crt = 96,~ \#D = 2745,\\ t_{\textup{off}} = 12.84,~ t_{\textup{on}} = 11.92,~ \BWE \approx 10^{-9} \end{matrix}$ 
\end{tabular}
\end{footnotesize}
\caption{Computational data and results for multi-graded and multi-unmixed examples. Here $A$ is as in \eqref{eq:parametersunmixed}.}
\label{tab:mgmu}
\end{table}

To conclude this subsection, we present a classical example of a
square mixed system in $n = 3$ variables coming from molecular biology
\cite[Sec.~3.3]{emiris1999computer}. The following code generates and
solves these equations:
\begin{verbatim}
@polyvar t[1:3]
b = [-13 -1 -1 24 -1; -13 -1 -1 24 -1; -13 -1 -1 24 -1]
mons1 = [1 t[2]^2 t[3]^2 t[2]*t[3] t[2]^2*t[3]^2]
mons2 = [1 t[3]^2 t[1]^2 t[3]*t[1] t[3]^2*t[1]^2]
mons3 = [1 t[1]^2 t[2]^2 t[1]*t[2] t[1]^2*t[2]^2]
f = [b[1,:]'*mons1';b[2,:]'*mons2';b[3,:]'*mons3'][:]
sol, A0, E, D = EigenvalueSolver.solve_CI_mixed(f,t)
\end{verbatim}
In this case, we obtain
$\delta = \gamma = \crt = 16,~ \#D = 200,~ t_{\textup{off}} = 0.53,~
t_{\textup{on}} = 0.02,~ \BWE \approx 10^{-13}$. The function
\texttt{certify} tells us that all 16 solutions are real, confirming
the observation made in \cite{emiris1999computer}.

\subsection{Overdetermined systems}
\label{sec:ex:over}
We now consider examples of \emph{overdetermined systems}, by which we
mean cases where $s > n$. We will limit ourselves to unmixed systems
and use Algorithm \ref{alg:getATunmixed} to find admissible tuples
leading to small Macaulay matrices. These systems arise, for instance,
in tensor decomposition problems \cite{telen2021normal}. We present
examples where $\gamma$ is significantly larger than $\delta$ and show
that, nevertheless, our algorithms successfully extract
$\delta < \gamma$ relevant eigenvalues and consistently return all
solutions of the input systems.

We observe that the Macaulay matrices constructed in this section are
smaller than the ones obtained using other symbolic-numerical
techniques as (sparse) resultants \cite{emiris_complexity_1996} or its
generalization \cite{massri2016solving}.
The admissible tuples used in those symbolic-numerical algorithms lead to multiplication operators, for which $\gamma = \delta$. As observed
in \Cref{rmk:nsolsSpecialBound}, our matrices $M_g$ are too large to be multiplication operators. The extra time needed for computing the eigenvalues of these larger matrices is negligible compared to the time won by computing $M_g$ from a smaller Macaulay matrix.

The overdetermined systems considered in this section are constructed as
follows. For a fixed number of variables $n$, number of solutions $\D$
and set of exponents $A$, we generate $\D$ random points
$\z_1, \ldots, \z_\D$ in $\C^n$ by drawing their coordinates from a
complex standard normal distribution. We construct a Vandermonde type
matrix $\texttt{Vdm}$ whose rows consist
of the vectors $\z_i^A/\lVert \z_i^A \rVert_2 , i = 1, \ldots,
\D$. The nullspace of $\texttt{Vdm}$ is computed using SVD and its
columns represent $s = \#A - \D$ polynomials $f_1, \ldots, f_s$ with
support $A$. If we do not pick too many points, we have that $s > n$ and the
solutions of $\F = (f_1, \ldots, f_s)$ are exactly the points
$\z_1, \ldots, \z_\D$.

\subsubsection{Dense, overdetermined systems} \label{subsec:denseod}
In this subsection, we consider dense overdetermined systems,
i.e.~$A_0 = \Delta_n \cap \N^n$ and $A = (d \cdot \Delta_n) \cap \N^n$
for some degree $d \in \N_{>0}$. The \emph{offline} computation uses
Algorithm \ref{alg:getATunmixed} to find an admissible tuple, as well
as a left nullspace, and then execute Algorithm \ref{alg:solve} from
line \ref{line:qr} on. The \emph{online} computation uses this
admissible tuple to execute Algorithm \ref{alg:solve} directly. This
means that the offline version uses an incremental strategy for
computing the left nullspace, while the online version works directly
with the large Macaulay matrix. The online version can be adapted to
work incrementally as well. We have chosen not to do this in order to
illustrate that, depending on $n, s$, the incremental approach may be
less or more efficient than the direct approach. In cases where the
incremental approach is more efficient, this may cause
$t_{\text{off}} < t_{\text{on}}$. In the square case $(s = n)$, this
happens for $n \geq 3$
\cite{mourrain2019truncated,parkinson2021analysis}, but our results
show that in the overdetermined case this might not happen. Further
research is necessary to make an automated choice. Table
\ref{tab:oddense} gives an overview of the computational results. The column indexed by $\widetilde{\# D}$ represents the size of the matrix that would be used in classical approaches. This is discussed in the final paragraph of this subsection.

The first 10 rows in Table \ref{tab:oddense} correspond to systems of
6 equations in 3 variables of increasing degree
$d = 2, 4, \ldots, 20$. Note that $\gamma > \delta$ for $d > 4$. In
all cases, $\delta$ distinct solutions were computed using our
algorithms and $\crt = \delta$. This means that exactly $\delta$ out
of $\gamma$ eigenvalues were selected and correctly processed to
compute solution coordinates. The maximum backward error grows faster
with the degree of the equations than for square systems
\cite{telen2018stabilized}. This can be remedied by using larger
admissible tuples to bring $\gamma$ closer to $\delta$, at the cost of
computing cokernels of larger matrices. However, our experiment shows
that we can find certified approximations for all 1765 intersection
points of 6 threefolds of degree 20 within less than 10 minutes. All
of these are within two Newton refinement steps from having a backward
error of machine precision.

The next 5 rows of Table \ref{tab:oddense} contain results for 18
dense equations in 6 variables of increasing degree
$d = 2, 3, \ldots, 6$. Note that $t_{\textup{on}} > t_{\textup{off}}$
for $d >2$. This is due to the incremental approach for the offline
phase, as mentioned above.

In the following 7 rows of Table \ref{tab:oddense}, we illustrate the
effect of increasing the number of variables when we fix the degree
$d = 3$. We work with overdetermined systems for which $s =
2n$. Although the complexity of eigenvalue methods usually scales
badly with the number of variables, these results show that when the
system is `sufficiently overdetermined', our algorithms can find
feasible admissible tuples to solve cubic equations in 8 variables in
no more than 20 seconds.

Finally, the last rows of Table \ref{tab:oddense} correspond to
systems of cubic equations in 15 variables with an increasing number
$\delta = 200, 300, \ldots, 600$ of solutions. Note that the
computation time \emph{decreases} with the number of solutions. The
reason is that for all these values of $\delta$, we can work with the
same support $D$ for the Macaulay matrix. This means that the matrix
has the same number of rows for each system. The number of columns,
however, depends on the number of equations, which increases with
decreasing $\D$ by construction. For $\delta = 700$, we need a larger
set of exponents $D$, causing memory issues.
\begin{table}[h!]
\centering
\begin{footnotesize}
\begin{tabular}{cccccccccccc}
$n$ & $s$ & $d$ & $\D$ & $\crt$ & $\gamma$  & $\#D$  & $\widetilde{\#D}$ & $\BWE$ & $\overline{\BWE}$ & $t_{\textup{off}}$ & $t_{\textup{on}}$ \\ \hline
3 & 6 & 2 & 4 & 4 & 4 & 10 & 10  & 5.75e-16 & 3.20e-16 & 1.25e-03 & 1.34e-03 \\ 
3 & 6 & 4 & 29 & 29 & 29 & 84 & 84  & 1.70e-14 & 2.54e-15 & 9.41e-03 & 6.33e-03 \\ 
3 & 6 & 6 & 78 & 78 & 100 & 220 & 286  & 7.07e-12 & 2.23e-14 & 7.00e-02 & 5.23e-02 \\ 
3 & 6 & 8 & 159 & 159 & 224 & 560 & 816  & 1.21e-12 & 4.67e-14 & 4.47e-01 & 2.90e-01 \\ 
3 & 6 & 10 & 280 & 280 & 465 & 969 & 1540  & 6.32e-10 & 6.63e-13 & 1.99e+00 & 1.32e+00 \\ 
3 & 6 & 12 & 449 & 449 & 820 & 1540 & 2600 & 5.09e-09 & 7.90e-12 & 8.76e+00 & 6.04e+00 \\ 
3 & 6 & 14 & 674 & 674 & 1280 & 2600 & 4495  & 1.51e-08 & 7.78e-12 & 3.88e+01 & 2.21e+01 \\ 
3 & 6 & 16 & 963 & 963 & 1938 & 3654 & 6545  & 3.57e-07 & 3.98e-11 & 1.26e+02 & 7.36e+01 \\ 
3 & 6 & 18 & 1324 & 1324 & 2776 & 4960 & 9139  & 1.83e-06 & 5.77e-10 & 3.54e+02 & 2.08e+02 \\ 
3 & 6 & 20 & 1765 & 1765 & 3780 & 7140 & 12341  & 1.11e-05 & 9.96e-10 & 9.85e+02 & 5.38e+02 \\  \hline
6 & 18 & 2 & 10 & 10 & 10 & 84 & 84  & 1.53e-14 & 2.96e-15 & 1.45e-02 & 8.82e-03 \\ 
6 & 18 & 3 & 66 & 66 & 66 & 462 & 462 & 4.51e-14 & 5.59e-15 & 1.69e-01 & 1.74e-01 \\ 
6 & 18 & 4 & 192 & 192 & 204 & 1716 & 3003 & 2.95e-12 & 6.36e-14 & 2.11e+00 & 3.79e+00 \\ 
6 & 18 & 5 & 444 & 444 & 1225 & 5005 & 8008 & 7.52e-12 & 1.76e-13 & 5.18e+01 & 7.86e+01 \\ 
6 & 18 & 6 & 906 & 906 & 4060 & 12376 & 27132 & 5.28e-10 & 2.37e-12 & 1.01e+03 & 1.33e+03 \\  \hline
2 & 4 & 3 & 6 & 6 & 6 & 10 & 10  & 3.02e-15 & 1.12e-15 & 1.15e-03 & 1.23e-03 \\ 
3 & 6 & 3 & 14 & 14 & 14 & 35 & 35  & 5.95e-15 & 1.49e-15 & 2.92e-03 & 2.35e-03 \\ 
4 & 8 & 3 & 27 & 27 & 27 & 126 & 126  & 3.85e-14 & 2.27e-15 & 1.56e-02 & 1.27e-02 \\ 
5 & 10 & 3 & 46 & 46 & 46 & 252 & 252  & 6.59e-14 & 8.89e-15 & 4.19e-02 & 2.04e-01 \\ 
6 & 12 & 3 & 72 & 72 & 126 & 462 & 924 & 3.70e-12 & 1.46e-13 & 1.61e-01 & 1.51e-01 \\ 
7 & 14 & 3 & 106 & 106 & 127 & 1716 & 3432 & 6.20e-12 & 3.96e-14 & 2.29e+00 & 4.26e+00 \\ 
8 & 16 & 3 & 149 & 149 & 483 & 3003 & 6435 & 8.31e-12 & 1.05e-13 & 1.16e+01 & 1.92e+01 \\  \hline
15 & 616 & 3 & 200 & 200 & 200 & 3876 & 3876 & 1.45e-13 & 1.04e-14 & 9.78e+01 & 5.80e+01 \\ 
15 & 516 & 3 & 300 & 300 & 300 & 3876 & 3876 & 3.66e-13 & 9.37e-15 & 8.25e+01 & 5.64e+01 \\ 
15 & 416 & 3 & 400 & 400 & 400 & 3876 & 3876 & 5.46e-13 & 1.44e-14 & 8.50e+01 & 5.42e+01 \\ 
15 & 316 & 3 & 500 & 500 & 500 & 3876 & 3876 & 4.25e-13 & 1.26e-14 & 6.38e+01 & 5.81e+01 \\ 
15 & 216 & 3 & 600 & 600 & 600 & 3876 & 3876 & 4.86e-13 & 1.41e-14 & 4.91e+01 & 4.65e+01 \\  
\end{tabular}
\end{footnotesize}
\caption{Computational results for overdetermined, dense systems. See Table \ref{tab:notation} for the notation.}
\label{tab:oddense}
\end{table}

All systems $(f_0,\F)$ appearing in Table \ref{tab:oddense} are
semi-regular*. By \Cref{thm:semireg}, the minimal value of $\lambda$
such that $((f_0,\F),A_0,(E_0^\lambda,\bm{E}^\lambda);D^\lambda)$ is
an admissible tuple is the degree $\lambda_{min}$ of the lowest-degree monomial with a non-positive coefficient in
\[\textup{ES}_{\Delta_n}(t) \, \prod_{i = 0}^s (1 - t^{d_i})
  = \frac{(1 - t) \, (1 - t^{d})^s}{(1 - t^n)}.
\]
To illustrate the gain of using
such a minimal $\lambda_{min}$, we included the number $\widetilde{\# D}$ which corresponds to the number of monomials in the $D^\lambda$ for the smallest $\lambda$ which gives $\gamma = \delta$. That is, the smallest $\lambda$ for which the matrices $M_g$ in our algorithm represent multiplication matrices. For $n = 3, s = d = 6$, $\lambda_{min}$ is $9$, and the admissible tuple has $\#D = \#(9 \, \Delta_2 \cap \Z^2) = 220$ lattice points. Multiplication matrices are obtained from $\#D = \#(10 \, \Delta_2 \cap  \Z^2) = 286$. 
To see the benefit of our incremental construction over the bounds from Table \ref{tbl:paramsTori}, note that case 1 gives $\#D = \#(34 \, \Delta_2 \cap \Z^2) = 7770$, and the Minkowski sum of the Newton polytopes (Table \ref{tbl:paramsTori}, case 5) gives $\#D = \#(37 \, \Delta_2 \cap \Z^2) = 9880$.

\subsubsection{Unmixed, overdetermined systems}

We now use our algorithms to solve overdetermined unmixed systems. The
results are summarized in Table \ref{tab:odunmixed}. First, we set
$n = 3$ and choose $\delta$ such that $s = 6$. We define $A_0$ as the
columns of
\[\begin{pmatrix}
2 & 2 & 0 & 1 & 1 & 0 & 0 & 0 \\ 1 & 1 & 0 & 0 & 1 & 0 & 1 & 0 \\ 0 & 2 & 1 & 0 & 2 & 1 & 2 & 0
\end{pmatrix}.\]
The support $A$ is obtained as
$A = (d \cdot \conv(A_0)) \cap \N^3$ for increasing values of $d$. The
conclusions are similar to those for the $n = 3$ experiments in the
previous subsection. Note that $d = 8$ is the only reported case for which one solution could not be certified. 

Next, we set $n = 15, \delta = 100$ and we define $A_0 = \{ 0, e_1, e_2, \ldots, e_{13}, e_{13}+e_{14}, e_{14}+e_{15} \}$, where $e_i$ is the $i$-th standard basis vector of $\Z^{15}$. We set $A = (2 \cdot \conv(A_0)) \cap \N^{15}$. There are 136 exponents in $A$, of degree at most 4. 
\begin{table}[h!]
\centering
\begin{footnotesize}
\begin{tabular}{ccccccccccc}
$n$ & $s$ & $d$ & $\D$ & $\crt$ & $\gamma$  & $\#D$ & $\BWE$ & $\overline{\BWE}$ & $t_{\textup{off}}$ & $t_{\textup{on}}$ \\ \hline
3 & 6 & 1 & 3 & 3 & 3 & 33 & 8.91e-16 & 5.36e-16 & 1.25e+00 & 7.59e-01 \\ 
3 & 6 & 2 & 27 & 27 & 27 & 165 & 2.71e-13 & 1.99e-14 & 1.96e+00 & 2.30e-02 \\ 
3 & 6 & 3 & 76 & 76 & 93 & 291 & 3.94e-12 & 8.52e-14 & 2.07e+00 & 9.36e-02 \\ 
3 & 6 & 4 & 159 & 159 & 216 & 708 & 8.53e-11 & 6.62e-13 & 3.27e+00 & 5.06e-01 \\ 
3 & 6 & 5 & 285 & 285 & 415 & 1405 & 1.99e-08 & 6.42e-12 & 6.25e+00 & 2.78e+00 \\ 
3 & 6 & 6 & 463 & 463 & 891 & 1881 & 2.00e-06 & 6.15e-11 & 1.56e+01 & 1.06e+01 \\ 
3 & 6 & 7 & 702 & 702 & 1387 & 3133 & 4.56e-05 & 7.06e-10 & 5.66e+01 & 4.51e+01 \\ 
3 & 6 & 8 & 1011 & 1010 & 2031 & 4845 & 9.29e-05 & 3.78e-10 & 1.86e+02 & 1.61e+02 \\  
 \hline
15 & 36 & 2 & 100 & 100 & 100 & 3876 & 2.88e-13 & 7.42e-15 & 8.07e+01 & 4.59e+01 \\ 
\end{tabular}
\end{footnotesize}
\caption{Computational results for overdetermined, unmixed systems. See Table \ref{tab:notation} for the notation.}
\label{tab:odunmixed}
\end{table}

\begin{remark}[Noisy coefficients]
  As an important direction for future research, we note that our
  eigenvalue algorithms can be used to compute `solutions' to
  overdetermined systems with noisy coefficients. For instance, the noise
  level needs to be taken into account when setting the relative
  tolerance for computing the left nullspace in line
  \ref{line:leftnull} of Algorithm \ref{alg:solve}. This is expected
  to work especially well for strongly overdetermined problems with
  only a few solutions.
\end{remark}

\subsection{Solutions at infinity} \label{subsec:solatinf}
An important feature of our algorithms is that they can deal with systems having isolated solutions \emph{at or near infinity}. To illustrate this, we work with the same set-up as in \Cref{subsec:denseod} with parameters $n = 7, d = 3$ and $s = 14$, implying $\delta = 106$. We generate 106 random complex points $\z_1, \ldots, \z_{106}$ as before, and then multiply the coordinates of $\z_{106}$ by a factor $10^e$ for increasing values of $e$. That is, we let one of 106 solutions drift off to infinity. Figure \ref{fig:infty} shows the maximal 2-norm of the computed solutions as well as the maximal backward error $\BWE$ for $e = 0, \ldots, 14$. The results clearly show that the accuracy is not affected by the `outlier' solution. As $e$ grows larger, the solution $\z_{106}$ corresponds to an isolated solution of the face system $\f_v$ with $v = (1,1,1,1,1,1,1)$, see Remark \ref{rmk:notMuchSolutionsAtInf}. For all considered values of $e$, our algorithm computed $\crt = \delta = 106$ distinct certified approximate solutions.

\begin{figure}[h!]
\footnotesize
\centering
\begin{tikzpicture}[scale = 0.9]

\begin{axis}[%
width=2.7in,
height=1.5in,
at={(0.772in,0.516in)},
scale only axis,
xmin=0,
xmax=14,
xlabel style={font=\color{white!15!black}},
xlabel={$e$},
ymode=log,
ymin=1e-14,
ymax=1e14,
yminorticks=true,
axis background/.style={fill=white}
]
\addplot [color=mycolor1, mark size=1.7pt, mark=*, mark options={solid, mycolor1}, forget plot]
  table[row sep=crcr]{%
0 4.141848883188113\\ 
1 19.58378771773134\\ 
2 360.74121706676124\\ 
3 2389.404833045611\\ 
4 28612.31152861531\\ 
5 318898.85002400284\\ 
6 3.618214070944617e6\\ 
7 3.3256171976338834e7\\ 
8 1.503077475184459e8\\ 
9 2.229607768704204e9\\ 
10 2.9545620092458317e10\\ 
11 2.923873172089561e11\\ 
12 1.187963945370828e12\\ 
13 1.771289701688117e13\\ 
14 2.290136098700917e13\\ 
};  \label{bluenorm}
\addplot [color=mycolor2, mark size=1.7pt, mark=*, mark options={solid, mycolor2}, forget plot]
  table[row sep=crcr]{%
0 9.512004946362486e-13\\ 
1 2.062377868870719e-12\\ 
2 5.079580198512017e-13\\ 
3 1.8643810148205277e-11\\ 
4 2.4122171130350622e-12\\ 
5 1.2173912973085575e-12\\ 
6 2.5119294477054433e-12\\ 
7 9.436566000413367e-13\\ 
8 4.945510386086103e-12\\ 
9 7.08162336920883e-13\\ 
10 3.4047717039410067e-12\\ 
11 3.2784507609689837e-12\\ 
12 6.149766931760818e-13\\ 
13 2.643728580354804e-12\\ 
14 8.705858658980046e-12\\ 
}; \label{orangeres}
\end{axis}
\end{tikzpicture}%
\vspace{-0.7 \baselineskip}
\caption{Max.~backward error $\BWE$ (\ref{orangeres}) and norm of the largest solution (\ref{bluenorm}) for the experiments in~Sec.~\ref{subsec:solatinf}.}
\label{fig:infty}
\end{figure}
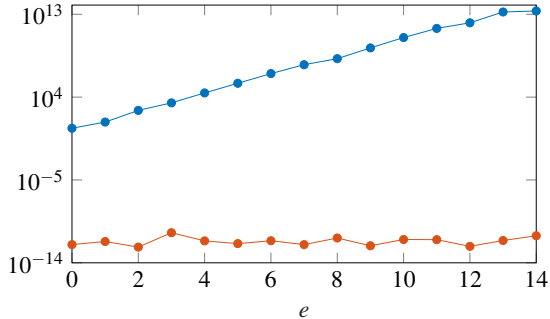

\subsection{Comparison with homotopy continuation methods}
\label{sec:ex:hom}
Homotopy continuation algorithms form another important class of numerical methods for solving polynomial systems \cite{sommese_numerical_2005}. These methods transform a \emph{start system} with known solutions continuously into the \emph{target system}, which is the system we want to solve, and track the solutions along the way. This process can usually only be set up for square systems, i.e.~$s = n$. In these cases, especially when $n = s$ is large $(\geq 4)$, homotopy continuation methods often outperform eigenvalue methods. When the system $\f$ is overdetermined ($s > n$), homotopy methods solve a square system $\f_{\text{square}}$ obtained by taking $n$ random $\C$-linear combinations of the $s$ input polynomials. The set of solutions of $\f$ is contained in the set of solutions of $\f_{\text{square}}$, so that the solutions of $\f$ can be extracted by an additional `filtering' step. Often $\f_{\text{square}}$ has \emph{many more} solutions than $\f$, so that many of the tracked paths do not end at a solution of $\f$. Below, we use the notation $\delta_{\text{square}}$ for the number of solutions of $\f_{\text{square}}$. 

Several implementations of homotopy methods exist, including Bertini \cite{bates2013numerically} and PHCpack \cite{verschelde1999algorithm}. Here, we choose to compare our computational results with the relatively recent Julia impementation \texttt{HomotopyContinuation.jl} \cite{breiding2018homotopycontinuation}. The motivation is twofold: it is implemented in the same programming language as \texttt{EigenvalueSolver.jl}, and it is considered the state of the art for the functionalities we are interested in. 
We point out that due to the extremely efficient implementation of
numerical path tracking in \texttt{HomotopyContinuation.jl}, the
package can outperform our eigenvalue solver even when
$\delta_{\text{square}}$ is significantly larger than $\delta$. For
instance, in the case $n = 3, d = 20$ from Table \ref{tab:oddense}, we
have $\delta_{\text{square}} = 8000 > \delta = 1765$, but
\texttt{HomotopyContinuation.jl} tracks all these 8000 paths in no
more than 40 seconds.  The performance is comparable for the row
$n = 6, d = 5$ in Table \ref{tab:oddense}, where
\texttt{HomotopyContinuation.jl} tracks
$\delta_{\text{square}} = 15625$ paths in about 45 seconds.
For all the above computations, we used the option
\texttt{start\_system = :total\_degree}, which is optimal for dense
systems and avoids polyhedral computations to generate start systems.

However, for strongly overdetermined systems, our algorithm outperforms
the homotopy approach. For example, for all the cases $n = 15, d = 3$,
Table \ref{tab:oddense} shows that our algorithms take no more than 2
minutes for $\delta \leq 600$. On the other hand, the number
$\delta_{\text{square}}$ equals $3^{15} = 14348907$, for which
\texttt{HomotopyContinuation.jl} shows an estimated duration of more
than 2 days.
Additionally, for the case $n = 15, d = 2$ in Table
\ref{tab:odunmixed}, we have $\delta_{\text{square}} = 32765$ and the
path tracking takes over 10 minutes, as compared to 48 seconds for the
online version of our algorithm and 65 seconds for the offline
version. In this last case we use the default \texttt{start\_system = :polyhedral}.

We conclude that for strongly overdetermined systems ($s \gg n$),
\texttt{EigenvalueSolver.jl} outperforms
\texttt{HomotopyContinuation.jl}, which suggests that eigenvalue methods
are more suitable to deal with this kind of systems.

\vspace{-.25\baselineskip}
{
\footnotesize
\paragraph{Acknowledgments}
Part of this work was done during the visit of the second author to TU
Berlin for the occasion of the MATH+ Thematic Einstein Semester on
Algebraic Geometry, Varieties, Polyhedra, Computation.
We thank the organizers of this nice semester for making this
collaboration possible.
We thank the anonymous reviewer for their useful comments
and constructive remarks. The first author was funded by the ERC under the European’s Horizon
2020 research and innovation programme (grant agreement No 787840). 
}

\vspace{-.75\baselineskip}
{
\footnotesize
\bibliographystyle{abbrv}
\setlength{\parskip}{0pt}
\bibliography{references}
}

\end{document}